\pdfoutput=1
\documentclass[11pt,a4paper]{article}
\usepackage[T1]{fontenc}
\usepackage[utf8]{inputenc}
\usepackage[british, french]{babel}
\usepackage{csquotes}

\usepackage[
backend=biber,
style=alphabetic,
citestyle=alphabetic
]{biblatex}

\usepackage{amsmath}
\usepackage{amsthm}
\usepackage{amsfonts}
\usepackage{amssymb}
\usepackage{hyperref}
\hypersetup{
    unicode=false,          
    pdftoolbar=true,        
    pdfmenubar=true,        
    pdffitwindow=false,     
    pdfstartview={FitH},    
    pdftitle={Forme modérément ramifiées de polydisques fermés et de dentelles},    
    pdfauthor={Marc Chapuis},     
    pdfnewwindow=true,      
    colorlinks=true,       
    linkcolor=red,          
    citecolor=green,        
    filecolor=magenta,      
    urlcolor=cyan           
}

\usepackage{color}
\usepackage{calc}
\usepackage{fancyhdr, layout}
\usepackage{enumerate}

\author{Marc \textsc{Chapuis}}
\title{Formes modérément ramifiées\\ de polydisques fermés et de dentelles}
\usepackage{tikz}
\usetikzlibrary{arrows,chains,matrix,positioning,scopes}
\makeatletter
\tikzset{join/.code=\tikzset{after node path={%
\ifx\tikzchainprevious\pgfutil@empty\else(\tikzchainprevious)%
edge[every join]#1(\tikzchaincurrent)\fi}}}
\makeatother
\tikzset{>=stealth',every on chain/.append style={join},
         every join/.style={->}}
\tikzstyle{labeled}=[execute at begin node=$\scriptstyle,
   execute at end node=$]

\usetikzlibrary{calc}

\newtheorem{theo}{Théorème}[section]
\newtheorem{prop}[theo]{Proposition}
\newtheorem{lemm}[theo]{Lemme}
\newtheorem{coro}[theo]{Corollaire}

\theoremstyle{definition}
\newtheorem{defi}[theo]{Définition}

\newtheorem*{nota}{Notation}

\theoremstyle{remark}

\newtheorem{rema}[theo]{Remarque}

\newcommand{\Gal}{\mathrm{Gal}}

\newcommand{\an}{\mathrm{an}}

\usepackage[sc]{mathpazo}
\usepackage[activate={true,nocompatibility},final,tracking=true,
kerning=true,spacing=true,factor=1100,stretch=15,shrink=15]{microtype}

\usepackage[top=2cm, bottom=3cm, left=2.25cm, right=2.25cm]{geometry}

\usepackage{float}
\usepackage[framemethod=tikz]{mdframed}

\mdfdefinestyle{myFigureBoxStyle}{
roundcorner=25pt,
tikzsetting={draw=gray, line width=1pt}}%

\makeatletter
  \newcommand\fs@myRoundBox{\def\@fs@cfont{\bfseries}\let\@fs@capt\floatc@plain
  \def\@fs@pre{\begin{mdframed}[style=myFigureBoxStyle]}%
  \def\@fs@mid{\vspace{\abovecaptionskip}}%
  \def\@fs@post{\end{mdframed}}\let\@fs@iftopcapt\iffalse}
\makeatother

\floatstyle{myRoundBox} 
\restylefloat{figure}

\usepackage[section]{placeins}

\makeindex

\begin{document}

\newpage
\maketitle

\begin{abstract}
Soit $k$ un corps ultramétrique complet, $L$ une extension galoisienne finie modérément ramifiée de $k$ et $X$ un espace $k$-analytique. Nous montrons que $X$ est isomorphe à un $k$-polydisque fermé (resp. une $k$-dentelle) si et seulement si $X_L$ est isomorphe à un $L$-polydisque fermé (resp. une $L$-dentelle) sur lequel l'action de $\Gal(L/k)$ est raisonnable.
\end{abstract}

\begin{otherlanguage}{british}
\begin{abstract}
Let $k$ be a complete non-Archimedean field, $L$ a finite tamely ramified galoisian extension of $k$ and $X$ a $k$-analytic space. We show that $X$ is isomorphic to a closed $k$-polydisc (resp. a $k$-lace) if and only if $X_L$ is isomorphic to a closed $L$-polydisc (resp. a $L$-lace) on which the action of $\Gal(L/k)$ is reasonable.
\end{abstract}
\end{otherlanguage}

\tableofcontents

\section*{Introduction}

Soient $k$ un corps ultramétrique complet, $X$ un espace $k$-analytique -- que nous entendons ici au sens de Berkovich \cite{Berkovich90, Berkovich93} -- et $L$ une extension galoisienne finie modérément ramifiée de $k$. Nous montrons les deux résultats suivants (les termes non standards seront clarifiés progressivement):
\begin{enumerate}
\item[{\hyperref[thm:forme polydisque]{(thm. \ref*{thm:forme polydisque})}}] 
$X_L$ est un $L$-polydisque fermé de polyrayon $\mathbf{r}=(r_1, ..., r_n)\in (\mathbb{R}_+^\times)^n$ sur lequel l'action de $\Gal(L/k)$ est résiduellement affine si et seulement si $X$ est un $k$-polydisque fermé de polyrayon $\mathbf{s}=(s_1, ..., s_n)$ avec, pour tout $i\in \lbrace 1, \dots, n \rbrace$
\[ 
  |L^\times|r_i = |L^\times| s_i. 
\]
\item[{\hyperref[theo:formes polycouronnes]{(thm. \ref*{theo:formes polycouronnes})}}] 
$X_L$ est une $L$-dentelle de type $U\subset (\mathbb{R}_+^\times)^n$, telle que l'action de $\Gal(L/k)$ sur $X_L$ est triviale sur le réseau si et seulement si $X$ est une $k$-dentelle de type $V$ et il existe $\mathbf{s}\in |L^\times|^n$ tel que
\[
  V = \lbrace (s_1r_1,\dots , s_nr_n), \mathbf{r}\in U\rbrace.
\]
\end{enumerate}
Le premier est l'analogue pour les \emph{polydisques fermés} du théorème d'Antoine Ducros pour les \emph{polydisques ouverts} \cite[thm. 3.5]{Ducros12}, déjà établi pour les \emph{disques fermés} par Tobias Schmidt, \cite[thm. 2.22]{Schmidt15}. Le second étend ce type de résultats à une classe d'espaces incluant les \textit{polycouronnes} que nous proposons d'appeler \emph{dentelles}. 

Nous entendons par une $k$-dentelle de type $U$ un espace $k$-analytique isomorphe à une partie de $\mathbb{G}_{m}^{n,\an}$ définie par $(|T_1|,\dots,|T_n|)\in U$ où $U$ est une partie $\mathbb{Z}$-linéaire par morceaux de $(\mathbb{R}_+^\times)^n$ non vide et connexe c'est-à-dire obtenue par recollement de $\mathbb{Z}$-polytopes de $(\mathbb{R}_+^\times)^n$ que l'on définit au moyen de formes affines dont la partie linéaire est à coefficients entiers et dont le terme constant appartient à $\mathbb{R}_+^\times$ (on voit ici $(\mathbb{R}_+^\times)^n$ comme un $\mathbb{R}$-espace vectoriel).  Nous donnons les définitions précises \hyperref[sec:polycouronnes]{section \ref*{sec:polycouronnes}}. Ces parties $\mathbb{Z}$-linéaires par morceaux sont un cas particulier des $c$-espaces linéaires par morceaux introduits par Berkovich dans \cite{Berkovich04}; le lecteur intéressé par de telles questions pourra aussi consulter \cite{Ducros12polytopes}. Notons entre autre que tous les produits d'intervalles sont des parties $\mathbb{Z}$-linéaires par morceaux (avec pour conséquence que toutes les $k$-polycouronnes sont des $k$-dentelles) et que tous les ouverts (connexes, non vides) de $(\mathbb{R}_+^\times)^n$ sont des parties $\mathbb{Z}$-linéaires par morceaux (connexes, non vides) de. 
En dimension un, les $k$-dentelles sont exactement les $k$-couronnes : que nous comprenons comme les espaces $k$-analytiques isomorphes à une partie de $\mathbb{G}_ {m,k}^{1,\an}$ définie par $|T|\in I$ où $I$ est un intervalle quelconque de $\mathbb{R}_+^\times$.\footnote{Un mot sur le choix de \emph{dentelle} (eng. \emph{lace}), les autres candidats considérés posent différents problèmes, par exemple : on a remarqué que les dentelles sont une classe plus large que les parties définies par $U$ un produit d'intervalles auxquels nous voulons réserver  \emph{polycouronne} ; \emph{polyèdre} colle bien quand $U$ est un polytope mais n'est pas très satisfaisant quand $U$ est ouvert ; l'adjectif \emph{cellulaire} est déjà utilisé en topologie, etc.

On peut parler de \emph{polydentelles} (eng. \emph{polylace}), mais en dimension un les dentelles sont exactement les couronnes, et il semble que le couple couronne/dentelle, plus expressif que dentelle/polydentelle, suffit à indiquer si l'on se restreint à la dimension un ou non : c'est le choix que nous avons fait ici.
}

Nous utilisons de manière fondamentale la théorie de réduction graduée introduite par Michael Temkin dans \cite{Temkin04}.  Nous renvoyons aux deux premières sections de \cite{Ducros12} pour une introduction détaillée au formalisme de l'algèbre commutative graduée et une reformulation de la théorie de ramification modérée que nous utiliserons. 

En dimension un, un avatar gradué du théorème $90$ de Hilbert a été  établi et utilisé dans sa preuve pour les disques fermés par Tobias Schmidt, \cite[Proposition 2.11]{Schmidt15}. Notre \hyperref[sec:Algèbre linéaire homogène]{première section} continue ce travail de traduction en termes gradués des outils classiques en en donnant un énoncé en dimension quelconque : \hyperref[pro:Hilbert 90 multiplicatif]{proposition \ref*{pro:Hilbert 90 multiplicatif}}.

Notre \hyperref[sec:polydisques]{deuxième section} traite des polydisques fermés. Nous suivons la stratégie de \cite{Ducros12} dont, comme Tobias Schmidt l'avait fait, nous remplaçons des arguments d'algèbre commutative locale par des calculs explicites permis par notre variante graduée du théorème $90$ de Hilbert. Or nous ne pouvons le faire que quand l'action de Galois est résiduellement affine, expliquons  ceci.

Nous partons d'un espace $k$-analytique $X$ qui devient un $L$-polydisque $X_L$ (de dimension $n$) auquel nous associons l'algèbre résiduelle graduée pour la norme spectrale de son algèbre des fonctions. Cette algèbre résiduelle est isomorphe à une algèbre polynomiale (graduée) $\widetilde{L}[\tau_1, \dots, \tau_n]$ et nous disons que l'action de Galois est résiduellement affine si l'action induite sur $\widetilde{L}[\tau_1, \dots, \tau_n]$ est affine (i.e. $g.(\tau_1,\dots, \tau_n)=A_g.(\tau_1,\dots,\tau_n)+(b_{1,g},\dots,b_{n,g})$). C'est toujours le cas en dimension un mais il n'y a pas de raison, \emph{a priori}, pour que ce soit vrai en dimension supérieure, même si nous ne connaissons pas d'exemple de forme modérément ramifiée de polydisque sur lequel l'action de Galois n'est pas résiduellement affine. Toutefois, si $X$ est un $k$-polydisque, l'action de Galois sur $X_L$ est toujours résiduellement affine.

La preuve marcherait sans cette hypothèse pour peu que l'action de Galois sur $\widetilde{L}[\tau_1, \dots, \tau_n]$ soit linéarisable. Cette question est liée à des problèmes compliqués de descriptions des automorphismes de l'espace affine (pour un exposé complet \cite{Kraft94}) : en dimension deux il sera peut-être possible, au prix d'un patient travail de vérification, d'adapter dans le formalisme de l'algèbre résiduelle graduée le résultat de \cite{Kambayashi75} sur les formes séparables du plan affine. Dès la dimension trois, même dans le cas non gradué, la question semble ouverte.

Dans la \hyperref[sec:polycouronnes]{dernière section}, nous traitons des dentelles sur le réseau desquelles Galois agit trivialement : la $L$-dentelle $X_L$ de type $U$ se rétracte par déformation sur une partie homéomorphe à $U$, son squelette analytique $S(X_L)^\an$, quand $U$ est d'intérieur non vide (en dimension un, un intervalle non réduit à un point), l'action de Galois est triviale sur le réseau si et seulement si elle est triviale sur $S(X_L)^\an$ (en dimension un si elle préserve l'orientation de $S(X_L)^\an$), et on généralise cette idée quand $U$ est d'intérieur vide.

Notre stratégie de preuve est la même que pour les polydisques à ceci près que celle-là repose sur une caractérisation de coordonnées par leurs réductions graduées pour \emph{la norme spectrale} et que cette approche échoue dès les couronnes (i.e les dentelles de dimension un). Nous résolvons la difficulté en considérant simultanément \emph{plusieurs} algèbres résiduelles.

Notons que Lorenzo Fantini et Daniele Turchetti ont développé dans \cite{FantiniTurchetti17} des techniques de descente galoisienne pour les espaces analytiques non archimédiens qui admettent un schéma formel spécial affine comme modèle au-dessus d'un anneau de valuation discrète complet. Avec ces techniques ils démontrent un cas particulier – sur un corps discrètement valué $k$, en dimension un, pour les couronnes de rayons distincts dans $\sqrt{|k^\times|}$ et des hypothèse plus strictes sur l'extension – de notre résultat sur les dentelles. Leur méthode complètement différente ouvre des pistes intéressantes quand la nôtre semble rester muette : ils obtiennent, par exemple, une description des formes de couronnes après extension quadratique sur lesquelles l'action de Galois permute l'orientation du squelette analytique.

\section{Algèbre linéaire homogène et Hilbert 90}
\label{sec:Algèbre linéaire homogène}

Dans tout ce qui suit, on désignera par $\Gamma$ un groupe abélien divisible, sans torsion et ordonné, toujours noté multiplicativement. On lui adjoint un élément absorbant pour la multiplication, on le note $0$ et c'est le plus petit élément. 

On désignera par $k$ un corps $\Gamma$-gradué.

\begin{rema} Avant de pouvoir traduire en termes graduées une preuve classique du théorème $90$ de Hilbert (telle qu'exposée par exemple dans \cite{Berhuy10}) il nous faut donner un sens précis aux espaces vectoriels et matrices que nous manipulons. C'est l'objet des deux premières sous-sections.
\end{rema}

\subsection{Espaces vectoriels homogènes}

\begin{defi} On appelle \emph{$k$-espace vectoriel homogène} un ensemble $V$ muni d'une décomposition $V=\coprod_{\rho\in \Gamma}V_\rho$, d'une structure de groupe abélien sur chacun des $V_\rho$, et d'une famille d'applications bilinéaires:
\[
  \bullet_{r, \rho} \, : \, k_r\times V_\rho\rightarrow V_{r\rho}, (\lambda, v)\mapsto \lambda \bullet_{r,\rho} v
\]
telles que  $1_1\bullet_{1,\rho} u = u$ et $(\lambda\mu)\bullet_{rs,\rho} u = \lambda \bullet_{r,\rho s}(\mu\bullet_{s,\rho} u)$ pour tout $(\lambda,\mu, u)$.

En particulier $V_\rho$ est un $k_1$-espace vectoriel.
\end{defi}

\begin{rema} Nous nous éloignons par cette définition des objets habituels de la réduction graduée au sens de Temkin, en effet un \emph{$k$-espace vectoriel gradué} y est muni d'une décomposition en somme directe $V=\bigoplus_{\rho\in \Gamma}V_\rho$ et non en réunion disjointe. Cela a pour conséquence que le corps gradué $k$, qui a une structure naturelle de $k$-espace vectoriel gradué, n'est pas un $k$-espace vectoriel homogène, ce rôle étant tenu par l'ensemble des éléments homogènes de $k$.

Si nous ne travaillions qu'avec des vecteurs nous nous en tiendrions aux espaces vectoriels gradués en précisant à chaque fois que nous travaillons avec des éléments homogènes, mais les matrices dont nous voulons parler correspondent naturellement aux morphismes d'espaces vectoriels homogènes et l'exposé perdrait nettement en clarté si nous nous refusions cet écart. 

En fait, on peut dès le départ définir tous les objets gradués comme réunions disjointes, c'est l'approche d'Antoine Ducros dans \cite[sec. 2.2]{DucrosEC} et alors les  notions d'espace vectoriel gradué et d'espace vectoriel homogène coïncident. Néanmoins écrire tout notre papier dans ce cadre nous couperait de \cite{Ducros12} et \cite{Schmidt15}.
\end{rema}

\begin{defi} Soit $U$ et $V$ deux $k$-espaces vectoriels homogènes, et $f: U \rightarrow V$ une application (dont on note $f_\rho$ la restriction à $U_\rho$), on dit que $f$ est  \emph{$k$-linéaire homogène} s'il existe $R_f \in \Gamma$ (que l'on appellera \emph{degré} de $f$) tel que pour tout $\rho \in \Gamma$:
\begin{itemize}
\item $\textrm{im}(f_{\rho})\subset V_{R_f\rho}$;
\item pour tout $u,v\in U_\rho$, $f_\rho(u+_\rho v) = f_\rho (u)+_{R_f\rho} f_\rho(v)$;
\item et pour tout $u$ dans $U_\rho$ et tout $\lambda\in k_r$, $f_{r\rho}(\lambda\bullet_{r,\rho} u)= \lambda \bullet_{r,R_f\rho} f_\rho(u)$.
\end{itemize}
En particulier $f_\rho : U_\rho \rightarrow V_{R_f\rho}$ est une application $k_1$-linéaire.
\end{defi}

\begin{defi} Soit $V$ un $k$-espace vectoriel homogène, soit $v_1, \dots, v_n$ une famille de vecteurs de degrés $\rho_1, \dots, \rho_n$, on dit que la famille $v_1, \dots, v_n$ est \emph{libre} si pour tout $R\in \Gamma$ et toute famille d'éléments $\lambda_i\in k_{R \rho_i^{-1}}$,
$$ \lambda_1 u_1 + \lambda_2 u_2 + \dots +\lambda_n u_n = 0_{R}$$
équivaut à $\lambda_i = 0_{R \rho_i^{-1}}$ pour tout $i\in \lbrace 1, \dots, n\rbrace$. 

On dit que la famille est \emph{génératrice} si pour tout vecteur $x$ de $V$, dont on note $R$ le degré, il existe une famille d'éléments $\lambda_i \in k_{R\rho_i^{-1}}$ telle que
$$x=\lambda_1 u_1 + \lambda_2 u_2 + \dots +\lambda_n u_n.$$

On dit que la famille est une \emph{base} si elle est libre et génératrice.
\end{defi}

\begin{rema} \label{rema:bases et produit vectoriel}
On peut vérifier que le théorème de la dimension ainsi que le théorème de la base incomplète sont encore vrais, i.e. deux bases quelconques d'un même espace vectoriel homogène ont même cardinal et tout espace vectoriel homogène admet une base; ceci nous permet de définir la \emph{dimension} d'un $k$-espace vectoriel homogène comme le cardinal commun à toutes ses bases.

On a alors une notion évidente de produit tensoriel de $k$-espace vectoriels homogènes. En effet il nous suffit prendre en compte la graduation dans les construction usuelles, par exemple: soient $U$ et $V$ deux $k$-espaces vectoriels homogènes, nous pouvons parler de l'espace vectoriel homogène libre $F(U\times V)$ engendré par $U\times V$ sur $k$ (où le degré d'un couple $(u,v)$ est défini comme le produit des degrés des coordonnées). Et alors les vecteurs du produit tensoriel que nous noterons $U\otimes_{k(1)} V$ sont les classes d'équivalences définis par les relations suivantes sur $F(A\times V)$:
\begin{itemize}
\item $\forall r_1, r_2, s \in \Gamma, u_1\in U_{r_1}, u_2\in U_{r_2}, v\in V_s$, $(u_1,v)+(u_2,v) \sim (u_1+u_2, v)$;
\item $\forall r, s_1, s_2 \in \Gamma, u_1\in U_{r}, v_1\in V_{s_1}, v_2\in V_{s_2}$, $(u,v_1)+(u,v_2) \sim (u,v_1+v_2)$;
\item $\forall r, s, \rho \in \Gamma, u\in U_{r}, v\in V_{s}, \lambda\in k_{\rho}$, $\lambda\bullet(u,v) \sim (\lambda\bullet u,v)$;
\item $\forall r, s, \rho \in \Gamma, u\in U_{r}, v\in V_{s}, \lambda\in k_{\rho}$, $\lambda\bullet(u,v) \sim (u,\lambda\bullet v)$.
\end{itemize}
Les opérations étant induites.
\end{rema}

\begin{defi}  Soit $U$ et $V$ deux $k$-espaces vectoriels homogènes, et $f:U\rightarrow V$ une application $k$-linéaire homogène, on appelera $\emph{noyau}$ de $f$ l'ensemble, que l'on notera encore $\ker (f)$, défini par
$$\ker (f) = \coprod_{\rho\in \Gamma} \ker (f_\rho).$$ 
\end{defi}

\begin{defi}
Soit $\mathbf{r}=(r_1, \dots, r_n)\in \Gamma^n$, l'ensemble
\[
\lbrace (x_1, ..., x_n) \, |\,  \exists R\in \Gamma, \forall i,  x_i \in k_{R r_i}\rbrace,
\]
a une  structure naturelle de $k$-espace vectoriel homogène, où le vecteur $(x_1,\dots, x_n)$ est de degré $R$.  On le notera $k_{\mathbf{r}}$.
\end{defi}

\begin{rema} Le $k$-espace vectoriel homogène $k_{(1)}$ n'est autre que l'ensemble des éléments homogènes de $k$ muni de la gradation et des opérations induites.
\end{rema}

\subsection{Matrices homogènes} 

\begin{defi}
Soit $\mathbf{r}=(r_1, \dots, r_n)\in \Gamma$ et $\mathbf{s}=(s_1, \dots, s_n)\in \Gamma$, on note:
$$\mathrm{M}\left(k, \mathbf{s}, \mathbf{r}\right) := \left\lbrace (a_{ij})_{i,j=1..n}\, | \, a_{ij} \in k_{s_i r_j^{-1}}\right\rbrace.$$
Cet ensemble, de \emph{matrices homogènes}, est en bijection avec les applications $k$-linéaires homogènes de $k_{\mathbf{r}}$ dans $k_{\mathbf{s}}$ de degré $1$. Pour les opérations évidentes,
\begin{itemize}
\item c'est un groupe additif;
\item pour tout $\mathbf{s}'=(s_1', \dots, s_n')\in \Gamma^n$ la multiplication à droite d'un élément de $\mathrm{M}\left(k, \mathbf{s},\mathbf{r}\right)$  avec un élément de $\mathrm{M}(k,\mathbf{r},\mathbf{s}')$ a un sens et définit un élément de $M(k,\mathbf{s},\mathbf{s}')$;
\item le déterminant de telles matrices a un sens, c'est un élément de $k_{\rho}$, où $\rho= \prod_{i=1..n}s_i r_i^{-1}$. Si le déterminant est non nul, la formule usuelle définit une matrice inverse qui appartient à $\mathrm{M}(k,\mathbf{s},\mathbf{r})$. Enfin on note $\mathrm{GL}(k,\mathbf{s},\mathbf{r})$ le sous-ensemble des matrices de déterminant non nul.
\end{itemize}
\end{defi}

\begin{rema} Insistons un moment sur le fait que  $\mathrm{GL}(k,\mathbf{s},\mathbf{r})$ puisse être vide, en effet, deux $k$-espaces vectoriels homogènes de même dimension ne sont pas forcément isomorphes. Prenons $\widetilde{\mathbb{Q}_{3}}_{(1)}$ et $\widetilde{\mathbb{Q}_{3}}_{(2)}$, c'est-à-dire deux fois le même ensemble – les éléments homogènes de $\widetilde{\mathbb{Q}_3}$ – où l'on voit (le vecteur) $\widetilde{1}_1$ comme étant de degré respectivement  $1$ et $\frac{1}{2}$. La seule application linéaire de degré $1$  de $\widetilde{\mathbb{Q}_{3}}_{(1)}$ dans $\widetilde{\mathbb{Q}_{3}}_{(2)}$ est l'application « nulle de degré 2 » c'est-à-dire donnée par la multiplication par $\widetilde{0}_2$, ou encore, traduit dans les notations que nous venons d'adopter,
\[
\mathrm{M}(\mathbb{Q}_3,(2),(1)) = \lbrace(\widetilde{0}_2)\rbrace \text{ et } \mathrm{GL}(\mathbb{Q}_3,(2),(1)) = \emptyset.
\]
\end{rema}

Fixons une dernière notation. 

\begin{defi}
L'ensemble $ \mathrm{M}(k,\mathbf{r},\mathbf{r}) $ des matrices qui correspondent aux endomorphismes $k$-linéaires homogènes de $k_{\mathbf{r}}$ (de degré $1$) forme un anneau, on le notera $$\mathrm{M}(k,\mathbf{r}),$$ et le sous-ensemble que l'on note de manière transparente $\mathrm{GL}(k,\mathbf{r})$ est donc toujours non vide (et forme un groupe multiplicatif).
\end{defi}

\begin{rema}  Ces ensembles ne dépendent pas tant de $(r_1, \dots, r_n)$ que de la droite vectorielle engendrée par $(r_1, \dots, r_n)$.

En dimension $1$, pour tout $r,s\in \Gamma$,
\begin{itemize}
\item $\mathrm{M}(k,(s),(r))= k_{s r^{-1}}$ et $\mathrm{GL}(k,(s),(r))=k_{s r^{-1}}^\times$;
\item $\mathrm{M}(k,(r)) = \mathrm{M}(k, (1)) = k_1$ et $\mathrm{GL}(k,(r)) = \mathrm{GL}(k, (1))= k_1^\times$.
\end{itemize}
\end{rema}

\subsection{Variantes de deux résultats classiques}

\begin{prop}[Indépendance linéaire des automorphismes de corps commutatifs gradués] 
\label{pro:independance lineaire des automorphismes}

Soit $k$ un corps commutatif gradué, soit une famille $(\chi_i)_{i=1\dots n}$ de $n$ automorphismes, deux à deux distincts, de $k$, soit $\rho\in \Gamma$ et soit $(a_i)_{i=1\dots n}$ une famille d'éléments de $k_\rho$, les assertions suivantes sont équivalentes:
\begin{itemize}
\item $a_1 \chi_1 + a_2 \chi_2+\dots+ a_n \chi_n = 0$;
\item pour tout $i \in\lbrace 1, \dots, n\rbrace$, $a_i = 0_\rho$.
\end{itemize}

\end{prop}

\begin{proof} Par récurrence sur $n\in \mathbb{N}^\times$. Si $n=1$, cela revient à montrer qu'un automorphisme $\chi$ de $k$ n'est pas nul, or $\chi(1_1)=1_1\neq 0_1$. Supposons $n\geq 2$, alors pour tout $x, y$ éléments homogènes de $k$ dont on note $r$ et $s$ les degrés
\[
  a_1\chi_1(xy)+a_2\chi_2(xy)+\dots+a_n\chi_n(xy) = 0_{\rho r s},
\]
soit
\[
  a_1\chi_1(x)\chi_1(y)+a_2\chi_2(x)\chi_2(y)+\dots+a_n\chi_n(x)\chi_n(y) = 0_{\rho r s}
\]
et donc pour tout $x$ de $k_r$,
\[
  a_1\chi_1(x)\chi_1+a_2\chi_2(x)\chi_2+\dots+a_n\chi_n(x)\chi_n = 0_{\rho r}.
\]
On a aussi pour tout $x$ de $k_r$, en multipliant $a_1 \chi_1 + a_2 \chi_2+\dots+ a_n \chi_n = 0$ par $\chi_1(x)$,
\[
  a_1\chi_1(x)\chi_1+a_2\chi_1(x)\chi_2+\dots+a_n\chi_1(x)\chi_n = 0.
\]
Par soustraction:
\[
  a_2 \left(\chi_1(x)-\chi_2(x) \right) \chi_2+\dots+a_n\left(\chi_1(x)-\chi_n(x)\right)\chi_n = 0.
\]
Cette somme comporte $n-1$ termes et comme $\chi_1\neq \chi_n$ on peut choisir $x$ de telle sorte que $\chi_1(x)\neq \chi_n(x)$. Par hypothèse de récurrence on obtient que $a_n =0_\rho$, donc que 
\[
  a_1 \chi_1 + a_2 \chi_2+\dots+ a_{n-1} \chi_{n-1} = 0,
\]
et, encore par hypothèse de récurrence, que $a_i =0_\rho$ pour tout $i\in\left\lbrace 1,\dots, n-1\right\rbrace$.
\end{proof}

\begin{defi} Soit $L/k$ une extension galoisienne de groupe de Galois $G$, et soit $U$ un espace vectoriel homogène sur $L$ avec une action $*$ de $G$ sur $U$. On note $\cdot$ l'action linéaire standard de $G$ sur $L$. On dit que $G$ agit par \emph{automorphismes semi-linéaires} sur $U$ si l'on a pour tout $r,s\in \Gamma$, $u, u'\in U_r, \lambda \in L_s$
$$\sigma*(u+u')= \sigma*u+\sigma*u'; $$
$$\sigma*(\lambda u) =(\sigma\cdot \lambda)  (\sigma*u).$$
\end{defi}

\begin{lemm} \label{lemm:extension du quotient}  Soit $U$ un espace vectoriel homogène sur $L$. Si $G$ agit sur $U$ par automorphismes semi-linéaires, alors $U^G := \lbrace u\in U \,|\, \sigma * u = u \text{ pour tout } \sigma \in G\rbrace$ est un $k$-espace vectoriel homogène, et l'application:
\[
  f: U^G\otimes_{k_{(1)}} L_{(1)}\rightarrow U, u\otimes \lambda \mapsto \lambda u
\]
est un isomorphisme homogène de degré $1$.
\end{lemm}

\begin{rema} Ici $U^G\otimes_{k_{(1)}} L_{(1)}$ désigne le $k$-espace vectoriel homogène obtenu comme produit vectoriel, cf. \hyperref[rema:bases et produit vectoriel]{remarque \ref*{rema:bases et produit vectoriel}}.
\end{rema}

\begin{proof} Il est clair que $U^G$ est un $k$-espace vectoriel homogène et que $f$ est $L$-linéaire homogène de degré $1$. Établissons d'abord la surjectivité de $f$. 

Soit $u$ un vecteur de $U$, soit $\lambda_1, \dots, \lambda_n$ une $k$-base de $L$  et soit $\sigma_1 = id_L, \sigma_2, \dots, \sigma_n$ les éléments de $G$. Posons pour $i\in\lbrace 1, \dots, n\rbrace$,
$$u_i = \sum_j \sigma_j * (\lambda_i u). $$
Pour tout $k\in\lbrace 1, \dots,n \rbrace$, on a
$$\sigma_k * u_i = \sum_j (\sigma_k\sigma_j)*(\lambda_i u). $$
Ainsi l'action de $\sigma_k$ sur $\sum_j \sigma_j * (\lambda_i u)$ ne fait que permuter les termes de la somme, donc $u_i \in U^G$.

Comme $\sigma_1, \dots, \sigma_n$ sont des $k$-automorphismes distincts de $L$, ils sont linéairement indépendants sur $L$ (\hyperref[pro:independance lineaire des automorphismes]{proposition \ref*{pro:independance lineaire des automorphismes}}). C'est pourquoi,  ayant, pour tout $i\in\lbrace 1, \dots, n\rbrace$, noté $r_i$ le degré de $\lambda_i$, la matrice $M = (\sigma_j \cdot \lambda_i)_{i,j}$ appartient à $\mathrm{GL}(L,(r_1,..., r_n),(1, \dots, 1))$. Puisque $G$ agit par automorphismes semi-linéaires, on a 
$$u_i = \sum_j \sigma_j * (\lambda_i u) =\sum_j(\sigma_j \cdot \lambda_i) (\sigma_j * u).$$
Maintenant si l'on écrit $M^{-1}= (m_{ij}')$, de $M^{-1}M= I_n$,
$$\sum_j m_{1k}'(\sigma_k\cdot\lambda_j)= \delta_{1k} \text{ pour tout } k=1\dots n.$$
Donc
$$\sum_j m_{ij}'u_j = \sum_j \sum_k m_{1j}'(\sigma_k\cdot \lambda_j) (\sigma_k * u)= \sum_k \delta_{1k}(\sigma_k*u) = \sigma_1 *u = u, $$
la dernière égalité découlant de $\sigma_1 = id_L$. Ainsi
$$u= \sum_j m_{1j}'u_j =f(\sum_j u_j\otimes m_{1j}'), $$
ce qui prouve la surjectivité de $f$.

Admettons  un instant que l'on ait montré que « toute famille de vecteurs $u_1, \dots, u_l \in~U^G$  $k$-linéairement indépendants est une famille de vecteurs $L$-linéairement indépendants dans $U$ ». Alors, soit $x\in \ker(f)$ (de degré $\rho$). On peut écrire:
$$x=u_1\otimes \mu_1 + \dots + u_n \otimes \mu_n, $$
où les $u_1, \dots,u_n$ sont $k$-linéairement indépendants (de degré $r_i$) et les $\mu_i$ sont des éléments homogènes de $L$ (où $\mu_i$ est de degré $\rho r_i^{-1}$). Par construction, $f(x)=0_\rho=\mu_1 u_1+\dots + \mu_n u_n$. Alors on a $\mu_i = 0_{\rho r_i^{-1}}$, et donc $x=0_\rho$, ce qui prouve l'injectivité de $f$.

Montrons maintenant ce que nous avions admis.  Supposons que l'on ait $l$ vecteurs $k$-linéairement indépendants $u_1, \dots, u_l\in U^G$ (de degrés $r_i$) pour lesquels il existe $\mu_1, \dots, \mu_l$ éléments homogènes de $L$ ($\mu_i$ est de degré $\rho r_i^{-1})$ non tous nuls tels que:
$$\mu_1 u_1+\dots+\mu_l u_l=0_\rho.$$
On peut supposer que $l$ est minimal, $l>1$ et $\mu_1=1_1$ (alors $\rho = r_1$). Par hypothèse, les $\mu_i$ ne sont pas tous dans $k$, donc on peut aussi supposer que $\mu_2\not\in k$. Soit $\sigma \in G$ tel que $\sigma \cdot \mu_2 \neq \mu_2$, alors
$$\sigma (\sum_i \mu_i u_i)= \sum_i(\sigma \cdot \mu_i) (\sigma*u_i) = \sum_i (\sigma \cdot \mu_i) u_i =0_\rho $$
et on obtient $\sum_{i\geq2}(\sigma \cdot \mu_i - \mu_i)u_i =0_\rho$, relation non triviale avec moins de termes, contradiction.
\end{proof}

\subsection{Hilbert 90 gradués} 

\begin{prop}[Hilbert 90 multiplicatif]
\label{pro:Hilbert 90 multiplicatif}
Soit $L/k$ une extension galoisienne de corps gradués dont on note $G$ le groupe de Galois, et soit $\mathbf{r}=(r_1,\dots, r_n)\in \Gamma^n$, alors pour tout 
$$\alpha\in \mathrm{Z}^1(G, \mathrm{GL}(L, \mathbf{r}))$$
il existe $\mathbf{s}=(s_1, \dots, s_n)\in \Gamma^n$ et $P\in \mathrm{GL}(L, \mathbf{s}, \mathbf{r})$ tels que
$$ \alpha(\sigma) = (\sigma \cdot P)  P^{-1} \text{ pour tout } \sigma\in G.$$
\end{prop}

\begin{proof}
Soit $\alpha\in Z^1(G, \mathrm{GL}(L,\mathbf{r}))$. On tord l'action naturelle de $G$ sur $L_{\mathbf{r}}$ en une action par automorphismes semi-linéaires:
$$\sigma*u=\frac{1}{\alpha(\sigma)} (\sigma \cdot u) \text{ pour tout } u\in L_{\mathbf{r}}, \sigma \in G.$$
D'après \hyperref[lemm:extension du quotient]{le lemme \ref*{lemm:extension du quotient}} il existe un isomorphisme $f: L_{\mathbf{r}}^G \otimes_{k_{(1)}} L_{(1)} \xrightarrow{\sim} L_{\mathbf{r}}$. En particulier $\dim_k(L_{\mathbf{r}}^G)= \dim_L (L_{\mathbf{r}}^G\otimes_{k_{(1)}} L_{(1)})= \dim_L(L_{\mathbf{r}})=n.$

Soit $v_1, \dots, v_n$ une $k$-base homogène de $L_{\mathbf{r}}^G$, c'est aussi une $L$-base homogène de $L_{\mathbf{r}}$, et, notant $s_i^{-1}$ le degré de $v_i$, la $j$-ème coordonnée de $v_i$ est de degré $s_i^{-1}r_j$. Ainsi la matrice $P$ dont les colonnes sont $v_1, \dots, v_n$ appartient à $\mathrm{GL}(L, \mathbf{r}, \mathbf{s})$.

Alors, pour tout $\sigma \in G$, la matrice $\sigma\cdot P$ est la matrice dont les colonnes sont $\sigma \cdot v_1, \dots, \sigma \cdot v_n$. Mais $v_1, \dots, v_n\in L_{\mathbf{r}}^G$, et donc
$$v_i=\sigma * v_i= \frac{1}{\alpha(\sigma)}(\sigma\cdot v_i) \text{ pour tout } i = 1\dots n. $$
Ou encore en termes de matrices:
\[ P = \frac{1}{\alpha(\sigma)} (\sigma \cdot P) \text{ pour tout } \sigma\in G. 
\qedhere
\]
\end{proof}

\begin{rema}
Cette version se ramène en dimension $1$ à l'énoncé ci-dessous, donnant une autre preuve de \cite[Proposition 2.11]{Schmidt15}.
\end{rema}

\begin{coro} 
\label{coro:Hilbert 90 multiplicatif en dimension 1}
Soit $L/k$ une extension galoisienne de corps gradués dont on note $G$ le groupe de Galois, alors pour tout $\alpha \in \mathrm{Z}^1\left(G, L_1^\times\right)$ il existe $s\in \Gamma$ et $\lambda \in L_s^\times$ tels que
\[
\alpha(\sigma) = \frac{\sigma.\lambda}{\lambda}  \text{ pour tout } \sigma\in G.
\]
En particulier, ayant noté $\mathfrak{D}(k)$ (resp. $\mathfrak{D}(L)$) le groupe des degrés non nuls de $k$ (resp. $L$),
$$\mathrm{H}^1\left(G, L_1^\times\right) = \mathfrak{D}(L)/\mathfrak{D}(k).$$
\end{coro}

\begin{proof}
Il suffit de rappeler que pour tout $r$ et $s$ dans $\Gamma$, sont isomorphes $L_1^\times$ et $\mathrm{GL}(L, (r))$, ainsi que $L_s^\times$ et $\mathrm{GL}(L, (s), (1))$.
\end{proof}

\begin{prop}[Hilbert $90$ additif gradué]\label{pro:Hilbert 90 gradué additif} Soit $k$ un corps gradué, $K$ une extension galoisienne de $k$, $r$ un élément de $\Gamma$, et $G=\Gal(K/k)$ le groupe de Galois de $K/k$, on a alors
\[
  \mathrm{H}^1\left(G,K_r\right) = 0.
\]
\end{prop}

\begin{proof} (D'après Antoine Ducros.) 
 On se ramène immédiatement au cas où $K$ est une extension finie de $k$. On commence par remarquer que comme $K$ est finie galoisienne sur $k$, il existe un élément $\lambda$ de trace $1$ dans $K^{\times}_1$. Il y a pour ce faire deux possibilités : décalquer la démonstration de la non dégénérescence de la trace dans le contexte gradué, en se ramenant par tensorisation avec une extension convenable au cas d'un produit fini d'extensions de $k$ ; ou bien utiliser le résultat connu pour $K_1$, et le fait que le noyau de la surjection $G \rightarrow \Gal(K_1/k_1)$ est d'ordre inversible dans $k$.

La multiplication par $\lambda$ est alors un endomorphisme de $K_r$ de trace égale à l'identité, et l'existence d'un tel endomorphisme permet de conclure (Cartan-Eilenberg-Serre utilise cette méthode dans Corps locaux).
\end{proof}

\section{Formes de polydisques fermés} 
\label{sec:polydisques}

À partir de maintenant lorsque nous parlerons d'espaces vectoriels ce seront toujours des espaces vectoriels gradués au sens de Temkin.

On note  $k$ un corps ultramétrique complet.

On note $T_1, \dots, T_n$ des coordonnées sur $\mathbb{A}_k^{n,\an}$.

\begin{defi} Soit $\mathbf{r} = (r_1, \dots, r_n) \in (\mathbb{R}^\times_+)^n$, nous noterons $\mathbb{D}_{\mathbf{r}}$ le domaine analytique de $\mathbb{A}_k^{n,\an}$ défini par les conditions $|T_i|\leq r_i$ pour tout $i\in \lbrace 1, ..., n\rbrace$, que nous appellerons $k$\emph{-polydisque fermé centré en l'origine de polyrayon $\mathbf{r}$}. 
 
Nous dirons qu'un espace $k$-analytique $X$ est un $k$\emph{-polydisque fermé de type} $\mathbf{r}$ s'il est isomorphe à $\mathbb{D}_{ \mathbf{r}}$.

Nous dirons qu'une famille $(f_1, \dots, f_n)$ de fonctions sur un $k$-polydisque fermé $X$ est un \emph{système de coordonnées} si les $f_i$ induisent un isomorphisme entre $X$ et un $k$-polydisque fermé centré en l'origine.
\end{defi}

\begin{defi}
Nous noterons $\mathcal{A}_{\mathbf{r}}$ l'algèbre des fonctions analytiques sur  $\mathbb{D}_{\mathbf{r}}$, que nous munissons de la norme spectrale $||.||_\infty$, et $\widetilde{\mathcal{A}_{\mathbf{r}}}$ sa réduction graduée pour cette norme.
\end{defi}

\begin{lemm}
\label{lemm:réduction de l'algèbre des fonctions d'un polydisque fermé}
Il existe un isomorphisme de $\widetilde{k}$-algèbres graduées entre $\widetilde{\mathcal{A}_{\mathbf{r}}}$ et l'anneau gradué 
\[
\widetilde{k}\left[\tau_1, \dots, \tau_n\right]
\]
modulo l'identification de $\widetilde{T}$ avec $\tau_i$.
\end{lemm}

Rappelons que l'anneau sous-jacent à cet anneau gradué est $\widetilde{k}\left[\tau_1, \dots, \tau_n\right]$ au sens usuel, et que le groupe de ses éléments homogènes de degré $s$ est formé, pour $s$ fixé, des polynômes de la forme $\Sigma_{I\in \mathbb{N}^n} a_I \tau^I$ avec $a_I$ homogène de degré $s\mathbf{r}^{-I}$ pour tout $I$, i.e. $\tau_i$ est de degré $r_i$.

\begin{proof}
Corollaire de \cite[Proposition 3.1 (i)]{Temkin04}.
\end{proof}

\begin{prop} 
\label{pro:système de coordonnées}
Soit $(f_1, \dots, f_n)$ une famille de fonctions analytiques sur $\mathbb{D}_{\mathbf{r}}$ telle que
$$
(\widetilde{f_i})=A(\tau_j)_{j=1..n}+B,
$$ 
où, notant $s_i=||f_i||_\infty$, $A\in \mathrm{GL}\left(\widetilde{k}, \mathbf{s}, \mathbf{r}\right)$ et $B \in \widetilde{k}_{\mathbf{s}}$.

Alors $(f_i)$ est un système de coordonnées induisant un isomorphisme entre $\mathbb{D}_{\mathbf{r}}$ et $\mathbb{D}_{\mathbf{s}}$.
\end{prop}

\begin{rema}
Rappelons que $A\in \mathrm{GL}\left(\widetilde{k}, \mathbf{s}, \mathbf{r}\right)$ signifie simplement que $A$ est de la forme $(a_{i,j})_{i,j=1..n}$ avec $a_{i,j}\in k_{r_j^{-1}s_i}$ et $\det(A)\in k_R^\times$ où $R=\prod_{i=1..n}(r_i^{-1}s_i)$, et que $B \in \widetilde{k}_{\mathbf{s}}$ signifie que $B$ est de la forme $(b_i)_{i=1..n}$ avec $b_i\in k_{s_i}$. 
\end{rema}

\begin{proof} Notons $\varphi: X \rightarrow Y$ le morphisme induit par les $f_i$. De même que dans \cite[prop. 3.4]{Ducros12}, quitte à composer $\varphi$ avec la translation par une pré-image dans $k^n$ de $-B$, on peut supposer que
\[
  \left(\widetilde{f_i}\right)_{i=1..n} = A.\left(\tau_i\right)_{i=1..n},
\]
et il existe pour tout $i$, $g_i$ dans $k[[\mathbf{T}]]$, telle que $g_i(f_i)=T_i$. Les arguments donnés dans sa preuve par Antoine Ducros nous permettent de garantir que les $g_i$ sont de norme spectrale majorée par $r_i$ et qu'elles convergent en particulier sur le polydisque \emph{ouvert}. En fait, les égalités $g_i(f_i)=T_i$ impliquent $\widetilde{g_i}(\widetilde{f_i})=\tau_i$ pour tout $i$, c'est-à-dire que 
$$\left(\widetilde{g_i}\right)_{i=1..n} = A^{-1}\left(\tau_i\right)_{i=1..n}.$$
Ainsi si l'on écrit $g_i =\sum b_{i,I}\mathbf{T}^I$, les monômes dont le multi-degré est supérieur à $2$ sont de norme strictement inférieure à celle de $g_i$, elle-même majorée par $r_i$, donc les monômes de multi-degré supérieur à $2$ sont de norme strictement majorée par $r_i$. Donc chacune des $g_i$ converge sur $Y$ et le $n$-uplet $(g_1, \dots, g_n)$ définit ainsi un morphisme $\psi : Y \rightarrow X$ qui vérifie $\psi \circ \varphi = \mathrm{Id}_X$ et $\varphi\circ \psi = \mathrm{Id}_Y$. Ainsi $\varphi$ est un isomorphisme.
\end{proof}

\begin{coro}
\label{coro:polydisques isomorphes et type}
Soit $\mathbf{r}=(r_1,\dots, r_n)\in (\mathbb{R}^\times_+)^n$ et $\mathbf{s}=(s_1,\dots, s_n)\in (\mathbb{R}^\times_+)^n$, deux $k$-polydisques fermés respectivement de type $\mathbf{r}$ et $\mathbf{s}$ sont isomorphes si et seulement si, quitte à réordonner les $s_i$, 
$$s_ir_i^{-1}\in|k^\times|$$
pour tout $i\in \lbrace 1, \dots, n \rbrace$.
\end{coro}

\begin{proof}
L'implication découle de la proposition ci-dessus, l'ensemble 
$$\mathrm{GL}\left(\widetilde{k}, \mathbf{s}, \mathbf{r}\right)$$
 est vide s'il n'est pas vrai que 
\[
\forall i\exists j, s_i r_j^{-1}\in |k^\times| \text{ et } \forall j\exists i, s_i r_j^{-1}\in |k^\times|.
\]
Considérons en effet une matrice $A\in\textrm{M}\left(\widetilde{k}, \mathbf{s}, \mathbf{r}\right)$. Le $j$-ème coefficient de la $i$-ème ligne est un élément de $k_{s_i r_j^{-1}}$ or $r_i^{-1}s_j\not\in |k^\times|$ implique $k_{s_i r_j^{-1}}=\lbrace 0 \rbrace$, ainsi si pour un $i$ il n'existe pas $j$ tel que $s_i r_j^{-1}\in |k^\times|$ la $i$-ème ligne ne comporte que des zéros et la matrice $A$ est de déterminant nul. De même si pour un $j$ il n'existe pas $i$ tel que $s_i r_j^{-1}\in |k^\times|$ la $j$-ème colonne ne comporte que des zéros.

Réciproquement, quitte à réordonner les $s_i$, il existe, pour tout $i\in\lbrace 1,\dots,n \rbrace$, $\lambda_i\in k^\times$ tel que $|\lambda_i|r_i = s_i$, donc si l'on prend un système de coordonnées $(T_1, \dots, T_n)$ induisant un automorphisme de $\mathbb{D}_{\mathbf{r}}$, d'après la proposition ci-dessus, $(\lambda_1 T_1, \dots, \lambda_n T_n)$ est encore un système de coordonnées sur $\mathbb{D}_{\mathbf{r}}$ induisant un isomorphisme vers $\mathbb{D}_{(|\lambda_i| r_i)}$, c'est-à-dire $\mathbb{D}_{\mathbf{s}}$.
\end{proof}

\begin{defi} Nous dirons que l'action d'un groupe $G$ sur un polydisque fermé $\mathbb{D}_\mathbf{r}$ est \emph{résiduellement affine} si l'action induite sur le plan $\widetilde{\mathcal{A}}_\mathbf{r} \simeq \widetilde{k}\left[\tau_1, \dots, \tau_n\right]$ est affine. C'est-à-dire, pour tout $g\in G$ :
$$ g.(\tau_i)= A(\tau_i)+B,$$
où, notant $s_i$ le degré de $g.\tau_i$, $A\in \mathrm{M}(\widetilde{k},\textbf{s}, \textbf{r})$ et $B\in \widetilde{k}_\mathbf{s}$.
\end{defi}

\begin{rema} 
La \hyperref[pro:système de coordonnées]{proposition \ref*{pro:système de coordonnées}} joue le même rôle dans notre preuve que \cite[Proposition 3.4]{Ducros12} et \cite[Proposition 2.18]{Schmidt15} dans leurs papiers. Elle implique en particulier qu'un groupe dont l'action sur $\mathbb{D}_\mathbf{r}$ induit une action par automorphismes \emph{affines} préservant le degré sur $\widetilde{\mathcal{A}}_\mathbf{r}$ agit par automorphismes préservant la norme spectrale sur $\mathbb{D}_\mathbf{r}$. 

Remarquons qu'une action de groupe par automorphismes préservant la norme spectrale sur $\mathbb{D}_\mathbf{r}$ est résiduellement affine quand pour tout $i$,
$$
|k^\times|r_i \cap |k^\times|\mathbf{r}^I \neq \emptyset
$$
implique que toutes les coordonnées de $I$ hormis la $i$-ème soient nulles. Ou, de manière équivalente dès la dimension deux, quand la famille $(r_i)$ est $\mathbb{Z}$-linéairement indépendante dans $\mathbb{R}_+^\times$ modulo $|k^\times|$.

Cela découle de la forme de $\widetilde{\mathcal{A}_{\mathbf{r}}}$, en effet c'est un fait classique en dimension un (l'ordre fini des éléments de $\Gal(\widetilde{L}/\widetilde{k})$ impliquant que le degré du polynôme $g.\widetilde{T}$ est $1$) et, en dimension supérieure, si les $r_i$ sont $\mathbb{Z}$-linéairement indépendants dans $\mathbb{R}_+^\times$ modulo $|k^\times|$, $\tau_i$ et $\tau_j$ ne peuvent apparaître dans la même somme homogène lorsque $i\neq j$, donc $\widetilde{\mathcal{A}_\mathbf{r}}=\widetilde{\mathcal{A}_{(r_1)}}\otimes_{\widetilde{k}} \widetilde{\mathcal{A}_{(r_2)}} \otimes_{\widetilde{k}} \dots \otimes_{\widetilde{k}} \widetilde{\mathcal{A}_{(r_n)}}$. Un automorphisme préservant la norme spectrale fixera alors $\widetilde{\mathcal{A}_{(r_i)}}$ pour tout $i$.

Il est intéréssant de noter que quand les $r_i$ sont $\mathbb{Z}$-linéairement indépendants dans $\mathbb{R}_+^\times$ modulo $|k^\times|$ et qu'aucun des $r_i$ n'est de torsion modulo $|k^\times|$, l'algèbre $\widetilde{\mathcal{A}_\mathbf{r}}$ est locale et qu'à cette proposition près la preuve d'Antoine Ducros dans \cite{Ducros12} marche directement. Il est amusant  qu'un polydisque fermé de polyrayon suffisamment « éloigné » de $|k^\times|^n$ (en dimension un c'est un disque qui a pour bord un point de type $3$) semble ainsi se comporter comme un polydisque ouvert.

En toute généralité le même argument ne peut marcher dès la dimension deux : prenons le bidisque de birayon $(1,1)$ sur un corps ultramétrique complet quelconque. La famille $(T_1, T_2+T_1^2)$ induit un automorphisme isométrique du bidisque – dont l'inverse est donné par $(T_1, T_2-T_1^2)$ – et a pour réduction $(\tau_1,\tau_2+\tau_1^2)$.
\end{rema}

\subsection{Deux lemmes d'extension des scalaires}

Il nous sera pratique d'avoir présenté sous la forme des deux lemmes ci-dessous des résultats établis au début de la preuve de \cite[Théorème 3.5]{Ducros12}. Nous en redonnons les démonstrations.

\begin{lemm}
\label{lemm:Extension des scalaires} 
Soit $k$ un corps ultramétrique complet et $X$ un espace $k$-analytique. Soit $L$ une extension finie de $k$. Notons $\mathcal{A}$ l'algèbre des fonctions analytiques sur $X$ et $\mathcal{B}$ l'algèbre des fonctions analytiques sur $X_L$. 
La flèche naturelle
\[
  \mathcal{A}\otimes_k L \rightarrow \mathcal{B} 
\]
est un isomorphisme. 
\end{lemm}

\begin{proof}
Supposons $L$ une extension finie de $k$. Choisissons une base $(l_1, \dots,l_n)$ de $L$ sur $k$. Si $V$ est un domaine affinoïde de $X$ d'algèbre $\mathcal{A}_V$ alors $V_L$ est un domaine affinoïde de $X_L$ d'algèbre $\mathcal{A}_V\otimes_k L\simeq \oplus \mathcal{A}_V.l_i$; il s'ensuit que la flèche naturelle 
\[
  \mathcal{A}\otimes_k L=\oplus \mathcal{A}.l_i\rightarrow\mathcal{B}
\]
est un isomorphisme.
\end{proof}

\begin{lemm}
\label{lemm:Extension des scalaires, réduction graduée, norme spectrale}
Soit $k$ un corps ultramétrique complet et $X$ un espace $k$-analytique compact. Soit $L$ une extension modérément ramifiée de $k$. Notons $\mathcal{A}$ l'algèbre des fonctions analytiques sur $X$ et $\mathcal{B}$ l'algèbre des fonctions analytiques sur $X_L$, que nous munissons de la norme spectrale, et $\widetilde{\mathcal{A}}$, $\widetilde{\mathcal{B}}$ leurs réductions pour cette norme.
Le morphisme canonique 
\[
\widetilde{\mathcal{A}}\otimes_{\widetilde{k}} \widetilde{L} \rightarrow \widetilde{\mathcal{B}}
\]
est un isomorphisme.
\end{lemm}

\begin{proof}
Choisissons un $d$-uplet $(l_1,\dots, l_d)$ d'éléments de $L^\times$ tel que $(\widetilde{l_1},\dots,\widetilde{l_d})$ soit une base de $\widetilde{L}$ sur $\widetilde{k}$. Soit $\mathbb{K}$ une extension ultramétrique complète quelconque de $k$. Comme $L$ est une extension modérément ramifiée de $k$, la famille $(\widetilde{l_1},\dots,\widetilde{l_n})$ est d'après \cite[§2.21]{Ducros12} encore une base de $\widetilde{(\mathbb{K}\otimes_k L)}$ sur $\widetilde{k}$. Si $(\lambda_1, \dots, \lambda_n)\in\mathbb{K}^n$, \cite[Lemme 2.3]{Ducros12} assure alors que 
\[ 
  ||\underset{ \in \mathbb{K}\otimes_k L } { \underbrace{\sum\lambda_i l_i}}|| = \max |\lambda_i|.|l_i|.
\]
Donc, par la définition de la norme spectrale d'une fonction analytique, on a pour tout $n$-uplet $(f_1,\dots, f_n)$ d'éléments de $\mathcal{A}$, l'égalité
\[ 
  ||\sum \underset{\in \mathcal{B}} { \underbrace {f_i.l_i} }||_\infty = \max ||f_i||_\infty |l_i|.
\]
On en déduit l'égalité des sommandes $\widetilde{\mathcal{B}}_\rho = \oplus_i \widetilde{L}_{\frac{\rho}{|l_i|}}\widetilde{l_i}$ et donc que le morphisme canonique $\widetilde{\mathcal{A}} \otimes_{\widetilde{k}} \widetilde{L} \rightarrow \widetilde{\mathcal{B}}$ est un isomorphisme.
\end{proof}

\begin{rema}
Soit $x$ un point de $X$ dont la fibre dans $X_L$ est réduite à un point $x_L$. Si l'on note $||.||_x$ et $||.||_{x_L}$ les semi-normes correspondantes sur $\mathcal{A}$ et $\mathcal{B}$. On a pour tout $n$-uplet $(f_1,\dots, f_n)$ d'éléments de $\mathcal{A}$, l'égalité
\[ 
  ||\sum \underset{\in \mathcal{B}} { \underbrace {f_i.l_i} }||_{x_L} = \max ||f_i||_{x} |l_i|.
\]
Donc si l'on note $\widetilde{\mathcal{A}}^x$, $\widetilde{\mathcal{B}}^{x_L}$ les réductions pour ces semi-normes. Le morphisme canonique 
\[
\widetilde{\mathcal{A}}^x\otimes_{\widetilde{k}} \widetilde{L} \rightarrow \widetilde{\mathcal{B}}^{x_L}
\]
est un isomorphisme.
\end{rema}

\subsection{Résultat}

\begin{theo}
\label{thm:forme polydisque} 
Soit $k$ un corps ultramétrique complet, soit $X$ un espace $k$-analytique et soit $L$ une extension galoisienne finie et modérément ramifiée de $k$. Alors $X_L$ est un $L$-polydisque fermé de polyrayon \(\mathbf{r}=(r_1,\dots,r_n)\) sur lequel l'action de Galois est résiduellement affine si et seulement si $X$ est un $k$-polydisque fermé de polyrayon $\mathbf{s}=(s_1, ..., s_n)$, avec pour tout $i\in \lbrace 1, \dots, n \rbrace$
\[
  |L^\times|r_i = |L^\times| s_i.
\]
\end{theo}

\begin{proof} 
Si $X$ est un $k$-polydisque fermé le résultat est immédiat.

Réciproquement. On désigne par $T_1, \dots, T_n$ un système de fonctions coordonnées sur $X_L$. Soit $\mathcal{A}$ (resp. $\mathcal{B}$) l'algèbre des fonctions analytiques sur $X$ (resp. $X_L$) que nous munissons de la norme spectrale, $||.||_\infty$.

Il existe (\hyperref[lemm:réduction de l'algèbre des fonctions d'un polydisque fermé]{lemme \ref*{lemm:réduction de l'algèbre des fonctions d'un polydisque fermé}}) un isomorphisme entre $\widetilde{\mathcal{B}}$ et $\widetilde{L}\left[\tau_1, ..., \tau_n\right]$ modulo lequel $\widetilde{T_i} =\tau_i$ est de degré $r_i$, pour tout $i\in \lbrace 1, \dots, n \rbrace$.

\begin{lemm} 
\label{lem:forme polydisque fermé} 
Quitte à remplacer $(T)_{i=1...n}$ par un autre système de coordonnées on peut supposer les $\tau_i$ invariants sous l'action de $\Gal(\widetilde{L}/\widetilde{k})$. 
\end{lemm}

\begin{proof} Par notre hypothèse sur l'action de Galois, pour tout $g\in \Gal\left(\widetilde{L}/\widetilde{k}\right)$:
\[
  g.(\tau)_{i=1..n} = A_g (\tau_i )_{i=1..n} +B_g
\]
avec
\begin{itemize}
\item \(A_g=(a_{ij,g})_{i,j=1..n} \in \mathrm{GL}\left(\widetilde{L},\mathbf{r}\right)\);
\item \(B_g=(b_{i,g})_{i=1..n}\in \widetilde{L}_{\mathbf{r}}\).
\end{itemize}
Considérons d'abord seulement les termes de degré $1$. On obtient l'identité:
\[
  A_{gh} = A_g \left(g.A_h\right).
\]
D'après notre variante du théorème 90 de Hilbert (\hyperref[pro:Hilbert 90 multiplicatif]{proposition \ref*{pro:Hilbert 90 multiplicatif}}), il existe  $\mathbf{s}=(s_i)_{i=1..n}$ dans $\mathbb{R}^n$ et $P$ dans $\rm{GL}\left(\widetilde{L},\mathbf{r},\mathbf{s}\right)$ telle que
$A_g = (g.P)P^{-1}$ pour tout $g\in \Gal(\widetilde{L}/\widetilde{k})$.

Pour tout $g\in \Gal(\widetilde{L}/\widetilde{k})$
\[
  g.(P^{-1} \left(\tau_i\right)_{i=1..n})= P^{-1} \left(\tau_i\right)_{i=1..n}+ B_g',
\]
avec $B_g'\in \widetilde{L}_{\mathbf{s}}$. Quitte à remplacer $(\tau_i)_{i=1..n}$ par $P^{-1} \left(\tau_i\right)_{i=1..n}$ on peut écrire pour tout $g \in  \Gal\left(\widetilde{L}/\widetilde{k}\right)$:
\[
  g.(\tau_i)_{i=1..n} = (\tau_i)_{i=1..n} + B_g,
\]
avec $B_g\in \widetilde{L}_{(s_i)}$. Soit coordonnée par coordonnée:
\[
  g.\tau_i= \tau_i + b_{i,g},
\]
où $b_{i,g}\in \widetilde{L}_{s_i}$. Fixons $i$, soient $g$ et $h$ deux éléments de $\Gal(\widetilde{L}/\widetilde{k})$; via les égalités
\[
  (gh).\tau_i=\tau_i+b_{i, gh}
\]
\[
  g.(h.\tau_i)=g.(\tau_i +b_{i,h})=\tau_i+b_{i,g}+g.b_{i, h}
\]
on obtient l'identité
\[
  b_{i, gh}=b_{i,g}+g(b_{i, h}).
\]
Ainsi, pour tout $i$, l'application $g\mapsto b_{i,g}$ est un cocycle additif de ${\rm H}^1\left(G, \widetilde{L}_{s_i}\right)$, et d'après la version graduée additive du théorème $90$ de Hilbert (\hyperref[pro:Hilbert 90 gradué additif]{proposition \ref*{pro:Hilbert 90 gradué additif}}), est un cobord, c'est-à-dire qu'il existe $\widetilde{\mu_i}$ dans $\widetilde{L}_{|{\lambda}^{-1}|R}$ tel que $b_{i,g}=g.\widetilde{\mu_i}-\widetilde{\mu_i}$ pour tout $g\in \Gal(\widetilde{L}/\widetilde{k})$.

Donc les $\tau_i$ peuvent être supposés invariants sous Galois, quitte à les translater par $\mu_i$. Il suffit alors de remplacer les $T_i$ par les antécédents de ces nouveaux $\tau_i$.
\end{proof}

Nous avons établi qu'il existe un isomorphisme entre $\widetilde{\mathcal{A}}\otimes_{\widetilde{k}}\widetilde{L}$, puisque isomorphe à $\widetilde{\mathcal{B}}$ (\hyperref[lemm:Extension des scalaires, réduction graduée, norme spectrale]{lemme \ref{lemm:Extension des scalaires, réduction graduée, norme spectrale}}), et
\[
  \widetilde{L}\left[\tau_1, ..., \tau_n\right],
\]
modulo lequel, pour tout $i$, $\widetilde{T_i} =\tau_i$ est de degré $s_i$. Comme $\widetilde{\mathcal{A}}= \left( \widetilde{\mathcal{A}}\otimes_{\widetilde{k}} \widetilde{L}\right)^{\Gal(\widetilde{L}/\widetilde{k})}$ on en déduit que $\widetilde{\mathcal{A}}$ est isomorphe à $\widetilde{k}\left[\tau_1, ..., \tau_n\right]$.

Il existe donc une famille de fonctions analytiques $(f_1, ..., f_n)$ dans $\mathcal{A}$ ayant pour réductions $(\tau_1, ..., \tau_n)$; en particulier, les $f_i$ sont de rayon spectral $s_i$. Soit $Y$ le $k$-polydisque de polyrayon  $\mathbf{s}$ et soit $\varphi$ le morphisme $X\rightarrow  Y$ induit par les $f_i$; il résulte de la \hyperref[pro:système de coordonnées]{proposition \ref*{pro:système de coordonnées}} que $\varphi_L : X_L \rightarrow Y_L$ est un isomorphisme.

Comme être un isomorphisme descend par extension quelconque des scalaires \cite[voir][Theorem 9.2]{Conrad10}, $\varphi$ est un isomorphisme.

De plus $X_L$ est un $L$-polydisque à la fois de type $\mathbf{r}$ et de type $\mathbf{s}$, donc, d'après le \hyperref[coro:polydisques isomorphes et type]{corollaire \ref*{coro:polydisques isomorphes et type}} quitte à réordonner les $s_i$, pour tout $i\in\lbrace 1, \dots, n \rbrace$,
\[
|L^\times|r_i = |L^\times|s_i. \qedhere
\]
\end{proof}

\section{Formes de dentelles}
\label{sec:polycouronnes}
On note $k$ un corps ultramétrique complet.

On note $T_1, \dots, T_n$ des coordonnées sur $\mathbb{G}_{m}^{n,\an}$. 

On note $\rho : x\mapsto \eta_{(|T_1(x)|,\dots,|T_n(x)|)}$ la rétraction de $\mathbb{G}_{m}^{n,\an}$ sur son squelette $S(\mathbb{G}_{m}^{n,\an})$.

\begin{nota}
Soit $U$ une partie de $(\mathbb{R}_+^\times)^n$, nous noterons
$$\eta_U =\lbrace \eta_{\mathbf{r}}\in \mathbb{G}_{m}^{n,\an}, \mathbf{r}\in U \rbrace ;$$ 
$$\mathbb{D}_U = \rho^{-1}(\eta_U) = \lbrace x\in \mathbb{G}_{m}^{n,\an}\, ,\, (|T_1(x)|,\dots,|T_n(x)|)\in U \rbrace.$$
\end{nota}

Nous considérons $(\mathbb{R}_+^\times)$ comme un espace vectoriel réel, la loi interne étant donnée par sa structure de groupe abélien et la loi externe par l'exponentiation coordonnée par coordonnée.

\begin{defi}
On note $\textrm{Aff}_{\mathbb{Z}}((\mathbb{R_+^\times})^n)$ l'ensemble des applications affines de $(\mathbb{R}_+^\times)^n$ vers $\mathbb{R}_+^\times$ de la forme 
$$(t_1,\dots, t_n)\mapsto r \prod t_i^{a_i}$$
où les $a_i$ appartiennent à $\mathbb{Z}$ et où $r\in \mathbb{R}_+^\times$.

Un \emph{$\mathbb{Z}$-polytope} de $(\mathbb{R}^\times_+)^n$ est une partie compacte de $(\mathbb{R}^\times_+)^n$ définie par une condition de la forme :
$$ 
\bigvee_{i\in I} \bigwedge_{j\in J} \varphi_{i,j}\leq 1
$$
où $I$ et $J$ sont des ensembles finis d'indices et où les $\varphi_{i,j}$ appartiennent à $\textrm{Aff}_{\mathbb{Z}}((\mathbb{R_+^\times})^n)$.
\end{defi}

\begin{defi}
Soit $X$ un espace topologique séparé et localement compact. Soit $Y$ une partie de $X$ et soit $(Y_i)$ une famille de sous-ensembles de $Y$. On dit que la famille $(Y_i)$ est un \emph{$G$-recouvrement} de $Y$ si tout point $y$ de $Y$ possède un voisinage dans $Y$ de la forme $U_{i\in I} Y_i$ où $I$ est un ensemble fini d'indices et où $y\in\cap_{i\in I} Y_i$.
\end{defi}

\begin{defi}
Une \emph{partie $\mathbb{Z}$-linéaire par morceaux de $(\mathbb{R}_+^\times)^n$} est une partie de $(\mathbb{R}_+^\times)^n$ qui admet un $G$-recouvrement par une famille de $\mathbb{Z}$-polytopes de $(\mathbb{R}^\times_+)^n$.
\end{defi}

\begin{rema} 
Soit $U$ un $\mathbb{Z}$-polytope non-vide convexe, c'est-à-dire non vide et défini par une condition de la forme :
$$ 
\bigwedge_{i=1..m} \varphi_{i}\leq 1
$$
où les $\varphi_{i}: \mathbf{t}\mapsto r_i \mathbf{t}^\mathbf{a_i}$ appartiennent à $\textrm{Aff}_{\mathbb{Z}}((\mathbb{R_+^\times})^n)$. Alors $\mathbb{D}_U$ est un domaine $k$-affinoïde de $\mathbb{G}_{m}^{n,\an}$. Pour le voir : noter $R$ un réel positif non nul tel que $U$ soit inclus dans le pavé $\subset [R^{-1};R]^n$ et remarquer que la partie $\mathbb{D}_U$  est isomorphe au spectre analytique de
$$k\lbrace R^{-1}T_1,R^{-1}T_1^{-1}, \dots R^{-1}T_n, R^{-1}T_n^{-1}, r_1^{-1}S_1,\dots, r_m^{-1}S_m \rbrace/(\mathbf{T}^\mathbf{a_i}-S_i)_{i=1..m}.$$
Ainsi quand $U$ est une partie $\mathbb{Z}$-linéaire par morceaux de $(\mathbb{R}_+^\times)^n$ non vide, $\mathbb{D}_U$ est un domaine $k$-analytique de $\mathbb{G}_{m}^{n,\an}$ et donc un espace $k$-analytique. En effet, on a par construction un G-recouvrement de $\mathbb{D}_U$ par des domaines $k$-affinoïdes de $\mathbb{G}_{m}^{n,\an}$ (\cite{Berkovich93}).
\end{rema}

\begin{rema}\label{rema:maximalité du squelette} Pour toute partie $\mathbb{Z}$-linéaire par morceaux de $(\mathbb{R}_+^\times)^n$ on peut décrire le squelette $\eta_U$ de $\mathbb{D}_U$ comme l'ensemble des points maximaux de $\mathbb{D}_U$ pour la relation : $ x \leq y $ si et seulement si $|f(x)| \leq |f(y)|$ pour toute fonction analytique $f$ sur $\mathbb{D}_U$.

Donc le squelette ne dépend pas du choix de coordonnées et  $\sigma(\eta_U)=\eta_U$ pour tout automorphisme $\sigma$ d'espace analytique de $\mathbb{D}_U$ ($\sigma$ peut agir non trivialement sur $k$).
\end{rema}

\begin{rema} Soit $U$ une partie $\mathbb{Z}$-linéaire par morceaux de $(\mathbb{R}_+^\times)^n$ connexe, la rétraction $\rho$ de $\mathbb{G}_{m}^{n,\an}$ sur son squelette induit une rétraction canonique de $\mathbb{D}_U$ vers $\eta_U$. En fait il existe \cite{Berkovich99} une rétraction par déformation $\Phi :\mathbb{G}_{m}^{n,\an}\times [0;1] \rightarrow S(\mathbb{G}_{m}^{n,\an})$ qui préserve $|T_i|$ pour tout $i$ et donc rétracte $\mathbb{D}_U$ sur $\eta_U$, de la connexité de $\eta_U$ découle alors celle de $\mathbb{D}_U$; la rétraction $\rho$ est simplement $\Phi(\bullet, 1)$.

Si $L$ est une extension complète de $k$ nous utiliserons les notations $\mathbb{D}_{U, L}$ et $\eta_{U, L}$ dans un sens évident. La flèche $\mathbb{D}_{U,L}\rightarrow\mathbb{D}_U$ induit un homéomorphisme $\eta_{U,L} \simeq \eta_U$, et, par les formules, $\mathbb{D}_{U, L}\rightarrow \mathbb{D}_{U}$ commute aux rétractions canoniques de $\mathbb{D}_{U, L}$ sur $\eta_{U,L}$ et de $\mathbb{D}_{U}$ sur $\eta_U$.

Le point $\eta_{\mathbf{r},\widehat{k^a}}$ de $\mathbb{G}_{\widehat{k^a}}^{n,\an}$ est par définition invariant sous l'action de Galois ; il s'ensuit que si $L$ est une extension complète de $k$ admettant un $k$-plongement isométrique dans $\widehat{k^a}$ (appellé \emph{extension presque algébrique de $k$} dans \cite{DucrosEC}), l'image réciproque de $\eta_U$ sur $\mathbb{D}_{U,L}$ est exactement $\eta_{U,L}$.
\end{rema}

\begin{defi} Nous dirons qu'un espace $k$-analytique $X$ est une \emph{$k$-dentelle de type $U$} s'il est isomorphe à $\mathbb{D}_U$ avec $U$ une partie $\mathbb{Z}$-linéaire par morceaux de $(\mathbb{R}_+^\times)^n$ connexe et non vide.

Nous dirons qu'une famille $(f_1, \dots, f_n)$ de fonctions sur une $k$-dentelle $X$ est un \emph{système de coordonnées} si les $f_i$ induisent un isomorphisme entre $X$ et $\mathbb{D}_U$ où $U$ est une partie $\mathbb{Z}$-linéaire par morceaux de $(\mathbb{R}_+^\times)^n$ connexe et non vide.

Un isomorphisme $X\simeq \mathbb{D}_U$ identifie $\eta_U$ à une partie de $X$ qui ne dépend pas du choix du système de coordonnées ; nous l'appellerons le squelette analytique de $X$ et nous le noterons $S^{\an}(X)$.
\end{defi}

\begin{rema} Quand $X$ est un espace $k$-analytique et $L$ une extension galoisienne de $k$ tels que $X_L$ soit une $L$-dentelle, l'action de $\Gal(L/k)$ sur $X_L$ fixe le squelette de $X_L$ ; c'est une conséquence de la \hyperref[rema:maximalité du squelette]{remarque \ref*{rema:maximalité du squelette}}.
\end{rema}

Désormais $U$ désignera toujours une partie $\mathbb{Z}$-linéaire par morceaux connexe et non vide.

\begin{defi} Nous noterons $\mathcal{A}_{U}$ l'algèbre des fonctions analytiques sur $\mathbb{D}_U$ et $\widetilde{\mathcal{A}_{U}}^\mathbf{r}$ sa réduction pour la semi-norme $||.||_{\mathbf{r}}$ évaluation en $\eta_\mathbf{r}$ pour tout $\mathbf{r}\in U$.

Quand $U$ est un singleton nous écrirons simplement $\widetilde{\mathcal{A}_{\lbrace r \rbrace}}$ au lieu de $\widetilde{\mathcal{A}_{\lbrace r \rbrace}}^\mathbf{r}$.
\end{defi}

\begin{rema} L'algèbre $\mathcal{A}_U$ est l'ensemble des séries $\sum_{I\in \mathbb{Z}^n} a_I \mathbf{T}^I$ telles que, pour tout $\mathbf{r}\in U$,
$$\lim_{|I|\rightarrow+\infty}|a_I|\mathbf{r}^I = 0.$$
\end{rema}

\begin{lemm} \label{lemm:réduction de l'algèbre des fonctions d'une polycouronne}
Il existe un isomorphisme de $\widetilde{k}$-algèbres graduées entre $\widetilde{\mathcal{A}_U}^\mathbf{r}$ et l'anneau gradué
$$\widetilde{k}[\tau_{\mathbf{r},1}, \tau_{\mathbf{r},1}^{-1}, \dots, \tau_{\mathbf{r},n}, \tau_{\mathbf{r},n}^{-1}]$$
modulo les identifications de $\widetilde{T_i}^\mathbf{r}$ avec $\tau_{\mathbf{r},i}$ (de degré $r_i$) et de $\widetilde{T_i^{-1}}^\mathbf{r}$ avec $\tau_{\mathbf{r},i}^{-1}$ (de degré $r_i^{-1}$).
\end{lemm}

\begin{proof}
Corollaire de \cite[Proposition 3.1 (ii)]{Temkin04} en remarquant que $||.||_{\mathbf{r}}$ est la norme spectrale de l'algèbre des fonction analytiques du polydisque de polyrayon $\mathbf{r}$.
\end{proof}

\subsection{Monôme strictement dominant et système de fonctions coordonnées} 

\begin{defi}
Soit $f=\sum_I a_I \mathbf{T}^I$ une fonction de $\mathcal{A}_U$. Nous dirons que $f$ a un \emph{monôme strictement dominant de polydegré $I$} si pour tout $\mathbf{r}\in U$ et tout $J\neq I$,
$$
|a_I| \mathbf{r}^I> |a_J|\mathbf{r}^J.
$$
\end{defi}

\begin{lemm} Une fonction de $\mathcal{A}_U$ est inversible si et seulement si elle a un monôme strictement dominant.
\end{lemm}
\begin{proof}
Si elle en a un elle s'écrit $f = a_I \mathbf{T}^I(1+u)$ avec $||u||_\mathbf{r}<1$ pour tout $\mathbf{r}\in U$ et est donc inversible. 

Réciproquement, $\widetilde{f}^{\mathbf{r}}$ est inversible dans $\widetilde{\mathcal{A}_{\lbrace \mathbf{r} \rbrace}}$. Donc s'écrit 
$\widetilde{a_{I(r)}}\tau_{\mathbf{r}}^{I(r)}$, c'est-à-dire que pour tout $\mathbf{r}\in U$ il existe un unique $n$-uplet $I(r)$ tel que
$$|a_{I(r)}|\mathbf{r}^{I(\mathbf{r})}=\max |a_{I}|\mathbf{r}^{I}.$$
La fonction $\mathbf{r}\mapsto I(\mathbf{r})$ est localement constante et donc constante sur $U$.
\end{proof}

\begin{lemm} 
\label{lemm:caractérisation fonctions coordonnées sur les polycouronnes}
Une famille de fonctions $(f_i)_{i=1..n}$ forme un système de fonctions coordonnées si et seulement si pour tout $i\in \lbrace 1, \dots, n\rbrace$, la fonction $f_i$ admet un monôme strictement dominant de polydegré $I_i$ et la matrice ayant pour lignes les $I_i$ appartient à $\mathrm{GL}_n(\mathbb{Z})$.
\end{lemm}

\begin{rema}
Par exemple sur la $k$-dentelle de type $\lbrace(1,1)\rbrace$, si l'on pose 
$$(f_1, f_2)=(T_1 T_2^{-1},T_2) \text{ et } (g_1,g_2)=(T_1 T_2, T_2)$$
on obtient deux systèmes de fonctions coordonnées induisant des automorphismes de la bicouronne (l'un étant inverse de l'autre).
\end{rema}

\begin{proof}
On notera $\mathbf{I}$ la matrice ayant pour lignes les $I_i$.

\textbf{Implication.} Chaque $f_i$ est inversible puisque fonction coordonnée, et admet donc un monôme strictement dominant de polydegré $I_i$. On vérifie que les polydegrés des monômes strictement dominant des fonctions induisant l'isomorphisme inverse donnent immédiatement l'inverse de $\mathbf{I}$.

\textbf{Réciproque.} Notons $J_i$ le $i$-ème vecteur ligne de $\mathbf{I}^{-1}$. Quitte à remplacer $(T_1,\dots, T_n)$ par $(\mathbf{T}^{J_1}, \dots, \mathbf{T}^{J_n})$ (système de fonctions coordonnées induisant un morphisme d'inverse évident), on peut supposer que $f_i$ a pour monôme strictement dominant $a_i T_i$. Donc, quitte à remplacer à nouveau $(T_1,\dots, T_n)$ par $(a_1^{-1} T_1, \dots, a_n^{-1} T_n)$, on peut supposer que $f_i$ a pour monôme strictement dominant $T_i$.

On peut écrire $f_i=T_i+f_i^{(1)}$ où $f_i^{(1)}\in \mathcal{A}_U$ et $||f_i^{(1)}||_{\mathbf{r}}<r_i$ pour tout $\mathbf{r}\in U$.

Montrons qu'une telle famille induit un automorphisme de $\mathbb{D}_{U}$. Il suffit qu'elles induisent un automorphisme de $\mathcal{A}_U$. Si nous arrivions à construire pour tout $i\in \lbrace 1, \dots, n \rbrace$ une fonction $g_i$ de $\mathcal{A}_U$ telle que $g_i(f_i)=T_i$ nous en aurions terminé. Nous allons construire $g_1$ par approximation successives. 

Si $f_1^{(1)}= 0$ la fonction $g_1^{(1)}=T_1$ convient, sinon fixons $\mathbf{r}\in U$, écrivons $f_1^{(1)}=\sum a_{i,I}^{(1)} \mathbf{T}^I$ et notons $\delta = \max_i(\frac{||f_i^{(1)}||_{\mathbf{r}}}{r_i})$ (par construction $\delta<1$). Remarquons que pour tout $I\in \mathbb{Z}^n$
$$||\mathbf{T}^I -\mathbf{f}^I||_{\mathbf{r}}=||\mathbf{T}^I -(\mathbf{T}-\mathbf{f}^{(1)})^I||_{\mathbf{r}}\leq\delta \mathbf{r}^I.$$
Ainsi, si l'on note
$$\mathcal{E} = \lbrace I \in \mathbb{Z}^n, \delta||f_1^{(1)}||_{\mathbf{r}}<|a_{i,I}| \mathbf{r}^I\leq ||f_1^{(1)}||_{\mathbf{r}} \rbrace$$
l'ensemble fini (puisque $f_1^{(1)}$ est un élément non nul de $\mathcal{A}_U$) des multi-indices des monômes de norme strictement supérieure à $\delta ||f_1^{(1)}||_{\mathbf{r}}$, il suit que
$$ ||f_1^{(1)}-\sum_{I\in \mathcal{E}}a_I\mathbf{f}^I||_{\mathbf{r}}\leq \delta||f_1^{(1)}||_{\mathbf{r}}.$$
Posant $f^{(2)}_1 = f_1'-\sum_{I\in\mathcal{E}} a_{1, I}\mathbf{f}^I$ et $g_1^{(2)} = g_1^{(1)} - \sum a_{1, I}\mathbf{T}^I$,
\begin{itemize}
\item $f_1^{(2)}$ appartient à $\mathcal{A}_U$ et $||f_1^{(2)}||_{\mathbf{r}}\leq \delta||f_1^{(1)}||_{\mathbf{r}}$
\item et $g^{(2)}_1(f_1, \dots, f_n) = T_1 + f_1^{(2)}$.
\end{itemize}
En itérant on obtient une suite $(f_1^{(n)})_{n\in\mathbb{N}}$ de fonctions de $\mathcal{A}_U$ qui tend vers $0$ et une suite $(g_1^{(n)})_{n\in\mathbb{N}}$ de fonctions de $\mathcal{A}_U$ dont on note $g_1$ la limite. Vérifions que $g_1\in \mathcal{A}_U$. Notons $b_I$ le coefficient de multi-indice $I$ de $g_1$ et $a_{1,I}^{(n)}$ celui de $f_1^{(n)}$. Soit $\varepsilon >0$
$$\lim_{i\rightarrow\infty} ||f_1^{(i)}||_\mathbf{r}=0 \Rightarrow  \exists m\in \mathbb{Z}, \forall i\geq m, \forall I \in\mathbb{Z}^n, |a_{1,I}^{(i)}|\mathbf{r}^I <\varepsilon,$$
et
$$\forall i\in \mathbb{N}, f_1^{(i)}\in \mathcal{A}_U \Rightarrow \forall \varepsilon>0, \exists N_i\in \mathbb{N}, \forall |I|\geq N_i, \forall I \in\mathbb{N}, |a_{1,I}^{(i)}|\mathbf{r}^I <\varepsilon.$$
Donc $|b_I|\mathbf{r}^I = \max_i |a_{1,I}^{(i)}|<\varepsilon$ pour tout  $|I|\geq \max_{i=1..m} (N_i)$, c'est-à-dire que $g_1\in \mathcal{A}_U$.
\end{proof}

\begin{defi}
Il suit de la discussion ci-dessus que le groupe 
$$\mathcal{L}(\mathbb{D}_U):=\mathcal{A}_U^\times/k^\times(1+\mathcal{A}_U^{\circ\circ})$$
est libre de rang $n$, et une famille de fonctions $(f_1, \dots, f_n)$ est un système de fonctions coordonnées si et seulement si elles sont inversibles et s'envoient sur une base de $\mathcal{L}(\mathbb{D}_U)$ (comme $\mathbb{Z}$-module). 

Nous appellerons \emph{réseau} de $\mathbb{D}_U$ le groupe $\mathcal{L}(\mathbb{D}_U)$.

Nous dirons qu'un automorphisme de $\mathbb{D}_U$ \emph{agit trivialement sur le réseau} s'il agit trivialement sur $\mathcal{L}(\mathbb{D}_U)$.
\end{defi}

\begin{rema}
\label{rema:fixer le résau, fixer le squelette}
Quand $U$ une partie $\mathbb{Z}$-linéaire par morceaux de $(\mathbb{R}_+^\times)^n$ connexe et non vide, et $L$ une extension galoisienne finie complète de $k$, l'action de Galois de $\Gal(L/k)$ sur $\mathbb{D}_{U, L}$, induite par son action naturelle sur $\mathbb{G}^{n,\an}_L$, est triviale sur le réseau : il suffit de prendre des coordonnées de $\mathbb{G}^{n,\an}_k$.

Ainsi, si un espace $k$-analytique $X$ est une $k$-dentelle, l'action de $\Gal(L/k)$ est triviale sur le réseau de $X_L$.

Quand $U$ est d'intérieur non vide un automorphisme de $\mathbb{D}_U$ agit trivialement sur le réseau si et seulement s'il agit trivialement sur le squelette analytique $\eta_U$ de $\mathbb{D}_U$.
\end{rema}

\subsection{Résultat}

\begin{nota}
Soit $U\subset(\mathbb{R}_+^\times)^n$ et $\mathbf{s}\in (\mathbb{R}_+^\times)^n$ on note $\mathbf{s}.U$ l'image de  $U$  via la translation multiplicative $\mathbf{r}\mapsto (s_1r_1,\dots, s_n r_n)$.
\end{nota}

\begin{theo}
\label{theo:formes polycouronnes}
Soit $k$ un corps ultramétrique complet, $L$ une extension galoisienne finie modérément ramifiée de $k$ et $X$ un espace $k$-analytique. Alors $X_L$ est une $L$-dentelle de type $U\subset (\mathbb{R}_+^\times)^n$ telle que l'action de $\Gal(L/k)$ est triviale sur le réseau de $X_L$ si et seulement si $X$ est une $k$-dentelle de type $\mathbf{r}.U$ avec $\mathbf{r}\in |L^\times|^n$.
\end{theo}

\begin{proof} 
Si $X$ est une $k$-dentelle de type $\mathbf{r}.U$, $X_L$ est une $L$-dentelle de type $\mathbf{r}.U$, donc  de type $U$ – puisque soit $(\lambda_i)_{i=1..n}$ une famille de $L^n$ avec $|\lambda_i|=r_i$, l'application $T_i\mapsto \lambda_i T_i$ induit un isomorphisme entre $\mathbb{D}_{U, L}$ et $\mathbb{D}_{\mathbf{r}.U,L}$ – et l'action de $\Gal(L/k)$ sur $X_L$ est triviale sur le réseau (\hyperref[rema:fixer le résau, fixer le squelette]{cf. remarque \ref*{rema:fixer le résau, fixer le squelette}}).

Réciproquement. On désigne par $(T_1, \dots, T_n)$ un système de fonctions coordonnées sur $X_L$ induisant un isomorphisme vers $\mathbb{D}_{U,L}$.
Soit $\mathcal{A}$ (resp. $\mathcal{B}$) l'algèbre des fonctions analytiques sur $X$ (resp. $X_L$). 

L'isomorphisme $X_L \rightarrow \mathbb{D}_{U,L}$ induit un homéomorphisme stable sous Galois entre $S^{\an}(X_L)$ le squelette analytique de $X_L$ et $\eta_{U,L}$. Le passage au quotient sous Galois induit dès lors un homéomorphisme entre une partie de $X$ et $\eta_{U,L}$. 

Pour tout $\mathbf{r}\in U$, nous notons naturellement $||.||_{\mathbf{r}}$ les semi-normes sur $\mathcal{A}$ et $\mathcal{B}$ correspondant à l'évaluation en $\eta_{\mathbf{r}, L}$ modulo les deux homéomorphismes mentionnés ci-dessus. Nous notons  $\widetilde{\mathcal{A}}^{\mathbf{r}}$ et $\widetilde{\mathcal{B}}^{\mathbf{r}}$ les réductions graduées de $\mathcal{A}$ et $\mathcal{B}$ pour les semi-normes $||.||_{\mathbf{r}}$.

Il existe (\hyperref[lemm:réduction de l'algèbre des fonctions d'une polycouronne]{lemme \ref*{lemm:réduction de l'algèbre des fonctions d'une polycouronne}}) un isomorphisme entre $\widetilde{\mathcal{B}}^{\mathbf{r}}$ et 
$$\widetilde{L}[\tau_{\mathbf{r},1}, \tau_{\mathbf{r},1}^{-1}, \dots, \tau_{\mathbf{r},n}, \tau_{\mathbf{r},n}^{-1}]$$
modulo les identifications de $\widetilde{T_i}^\mathbf{r}$ avec $\tau_{\mathbf{r},i}$ (de degré $r_i$) et de $\widetilde{T_i^{-1}}^\mathbf{r}$ avec $\tau_{\mathbf{r},i}^{-1}$ (de degré $r_i^{-1}$).

Par hypothèse l'action de Galois préserve le réseau de la couronne, autrement dit, pour tout $g\in \Gal(L/k)$ et tout $i\in \lbrace 1, \dots, n\rbrace$ la fonction $g.T_i$ a pour coefficient dominant $\alpha T_i$ avec $|\alpha|=1$ (puisque les $g.T_i$ induisent un automorphisme). Ainsi, pour tout $g\in \Gal(\widetilde{L}/\widetilde{k})$, tout $\mathbf{r}\in U$ et tout $i\in \lbrace 1, \dots, n\rbrace$ 
$$g.\tau_{\mathbf{r},i}= a_{i,g}\tau_{\mathbf{r},i}$$
où $a_{i,g}\in \widetilde{L}_1^\times$ (car, notant $m$ l'ordre de $g$ et $\alpha$ le degré de $a_{i,g}$ l'identité $g^m.(\tau_{\mathbf{r},i}=\tau_{\mathbf{r},i}$ implique que $\alpha^m = 1$ d'où $\alpha=1$). Les $a_{i,g}$ ne dépendent pas de $\mathbf{r}$ puisque obtenus par réduction du coefficient du même monôme strictement dominant.

La situation est exactement analogue au cas des polydisques. Pour tout $i\in \lbrace 1, \dots, n\rbrace$ l'application $g\mapsto a_{i,g}$ est un élément de $\textrm{Z}^1\left(\Gal(\widetilde{L}/\widetilde{k}),\widetilde{L}_1^\times\right)$ donc, d'après la version graduée du théorème $90$ de Hilbert (\hyperref[coro:Hilbert 90 multiplicatif en dimension 1]{corollaire \ref*{coro:Hilbert 90 multiplicatif en dimension 1}}), il existe $\lambda_i \in L^\times$ (il n'y a pas de raison que $|\lambda_i|=1$) tel que $a_{i,g}=\widetilde{\lambda_i}^{-1}(g.\widetilde{\lambda_i})$ pour tout $g\in \Gal(\widetilde{L}/\widetilde{k})$ et tout $i\in \lbrace 1, \dots, n\rbrace$.

Notons $\Lambda = (|\lambda_1|, \dots, |\lambda_n|)$. Les applications $T_i\mapsto \lambda_i T_i$ induisent un isomorphisme entre $\mathbb{D}_U$ et $\mathbb{D}_{\Lambda.U}$ (d'inverse évidemment induit par $T_i \mapsto \lambda_i^{-1} T_i$) donc, quitte à remplacer $U$ par $\Lambda.U$, on peut supposer, pour tout $g\in \Gal(\widetilde{L}/\widetilde{k})$, tout $\mathbf{r}\in U$ et tout $i\in \lbrace 1, \dots, n\rbrace$ que
$$g.\tau_{\mathbf{r},i}= \tau_{\mathbf{r},i}.$$

Choisissons un $d$-uplet $(l_1,\dots, l_d)$ d'éléments de $L^\times$ tel que $(\widetilde{l}_1, \dots, \widetilde{l}_d)$ soit une base de $\widetilde{L}$ sur $\widetilde{k}$, comme il ne nous coute rien de prendre $l_1=1$ nous le ferons. Les flèches naturelles 
$$\mathcal{A}\otimes_k L = \oplus \mathcal{A}.l_i \rightarrow \mathcal{B}$$
et, pour tout $\mathbf{r}\in U$,
$$\widetilde{\mathcal{A}}^{\mathbf{r}}\otimes_{\widetilde{k}} \widetilde{L} = \oplus \widetilde{\mathcal{A}}^{|l_i|^{-1}\mathbf{r}}.\widetilde{l_i} \rightarrow \widetilde{\mathcal{B}}^{\mathbf{r}}$$
sont des isomorphismes. Le premier vient du \hyperref[lemm:Extension des scalaires]{lemme \ref*{lemm:Extension des scalaires}}. Le deuxième du \hyperref[lemm:Extension des scalaires]{lemme \ref*{lemm:Extension des scalaires, réduction graduée, norme spectrale}} et de la remarque qui le suit.

Pour tout $i\in\lbrace 1,\dots, n\rbrace$, on peut donc écrire $T_i = \sum_{j=1..d} f_{i,j} l_j$ avec $f_{i,j}\in \mathcal{A}$ pour tout $j\in\lbrace 1,\dots, d\rbrace$. En passant à la réduction on obtient, pour tout $\mathbf{r}\in U$ et tout $i\in\lbrace 1,\dots, n\rbrace$,
$$
\tau_{\mathbf{r},i} = \widetilde{f_{i,1}}^\mathbf{r} + \sum_{j\neq 1} (\widetilde{f_{i,j}}^\mathbf{r})_{|l_j|^{-1}r_i}.\widetilde{l_j}
$$
où $\tau_{\mathbf{r},i}$, $\widetilde{f_{i,1}}^{\mathbf{r}}$ et tous les $(\widetilde{f_{i,j}}^\mathbf{r})_{|l_j|^{-1}r_i}$ (c'est-à-dire la composante de degré $|l_j|^{-1}r_i$ de la réduction de $f_{i,j}$ dans $\widetilde{\mathcal{A}}^\mathbf{r}$) sont invariants sous Galois et $(\widetilde{1},\widetilde{l_2},\dots,\widetilde{l_d})$ forme une base de $\widetilde{\mathcal{B}}^\mathbf{r}$ sur $\widetilde{\mathcal{A}}^\mathbf{r}$. On a égalité des termes Galois invariants, donc $\tau_{\mathbf{r},i} = \widetilde{f_{i,1}}^{\mathbf{r}}$.

Ainsi il existe une famille $(f_1, \dots, f_n)$ de fonctions de $\mathcal{A}$ (on a écrit $f_i$ pour $f_{i,1}$) telles que, après extension des scalaires, on ait, pour tout $i \in \lbrace 1, \dots, n \rbrace$, et tout $\mathbf{r}\in U$,
$$
\widetilde{f_i}^\mathbf{r} = \tau_{\mathbf{r},i}.
$$
Ou encore, pour tout $i\in\lbrace 1,\dots, n\rbrace$, dans $\mathcal{B}$, $f_i$ a pour monôme strictement dominant $a_iT_i$ avec $|a_i|=1$ et $\widetilde{a_i}=\widetilde{1}$. En particulier, d'après le \hyperref[lemm:caractérisation fonctions coordonnées sur les polycouronnes]{lemme \ref*{lemm:caractérisation fonctions coordonnées sur les polycouronnes}}, les $(f_i)$ forment un système de fonctions coordonnées sur $X_L$ (qui induit un automorphisme).

Il existe donc une famille de fonctions analytiques $(f_1, ..., f_n)$ dans $\mathcal{A}$ qui induisent un morphisme $\varphi : X\rightarrow \mathbb{D}_U$ tel que $\varphi_L : X_L \rightarrow \mathbb{D}_{U, L}$ soit un isomorphisme.
 
Comme être un isomorphisme descend par extension quelconque des scalaires \cite[voir][Theorem 9.2]{Conrad10}, $\varphi$ est un isomorphisme.
\end{proof}

\subsection{Exemple de forme non triviale de couronne}

Il était nécessaire que $\eta_{U,L}$ soit invariant sous l'action de Galois; donnons une forme modérément ramifiée de couronne qui n'est pas une couronne quand ce n'est pas le cas. Prenons comme corps de base $k=\mathbb{Q}_3$, comme extension modérément ramifiée $L=\mathbb{Q}_3[i]$ et intéressons-nous à la partie de $\mathbb{P}_{\mathbb{Q}_3}^{1,\an}$ définie par:
\[
	\lbrace x\in \mathbb{P}_{\mathbb{Q}_3}^{1,\an}, \frac{1}{4}\leq|(T^2+1)(x)|  \rbrace
\]
que l'on note $X$ (voir \autoref{fig:Forme non triviale de couronne}, le segment $\eta_{[0;\infty]}$ donné comme repère ne joue pas de rôle particulier). L'espace $\mathbb{Q}_3$-analytique $X$ n'est pas une $\mathbb{Q}_3$-couronne (la partie fermée ainsi définie aurait deux bouts de type $2$ ou $3$ et non un) et pourtant, après extension des scalaires, on peut écrire
\[
	X_{\mathbb{Q}_3[i]}=\lbrace x\in \mathbb{P}_{\mathbb{Q}_3[i]}^{1,\an}, \frac{1}{2}\leq|\frac{T+i}{T-i}(x)|\leq 2\rbrace
\]
et remarquer que l'on obtient une $X_{\mathbb{Q}_3[i]}$-couronne avec $f=\frac{T+i}{T-i}$ comme fonction coordonnée. Notant $\sigma$ la conjugaison complexe, $\sigma(f)=-f$ et l'action de Galois sur le réseau de $X_{\mathbb{Q}_3[i]}$ n'est pas triviale.

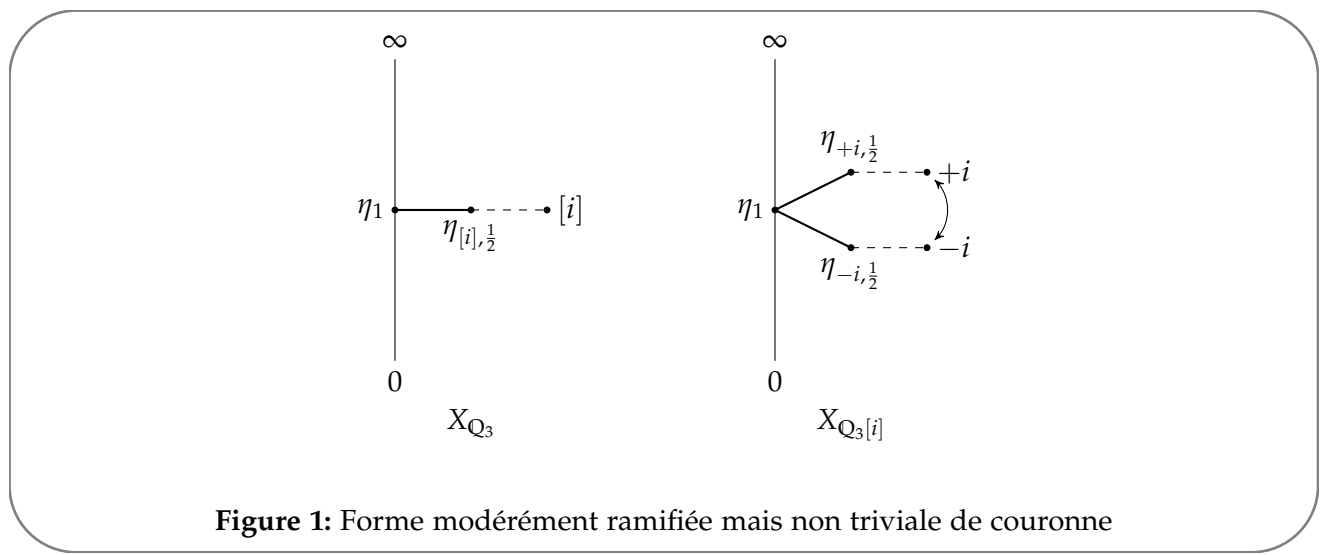
\begin{figure}[htb]
  \begin{center}
    \begin{tikzpicture}[scale=1]
      \draw (0,-2) -- (0,2);
      \node[below] at (0,-2) {$0$};
      \node[above] at (0,2) {$\infty$};
      
      \node[left] at (0,0) {$\eta_{1}$}; 
      \draw[black,fill=black] (0,0) circle (.1em);     
      
      \draw[thick] (0,0) -- (1,0);
      \draw[dashed] (1,0) -- (2,0);
      \draw[black,fill=black] (1,0) circle (.1em);
      \node[below] at (1,0) {$\eta_{[i],\frac{1}{2}}$}; 
      \draw[black,fill=black] (2,0) circle (.1em);
      \node[right] at (2,0) {$[i]$}; 
      
      \node[below] at (1,-2.5) {$X_{\mathbb{Q}_3}$};
      
      \begin{scope}[shift={(5,0)}]
      \draw (0,-2) -- (0,2);
      \node[below] at (0,-2) {$0$};
      \node[above] at (0,2) {$\infty$};
      \node[left] at (0,0) {$\eta_{1}$}; 
      \draw[black,fill=black] (0,0) circle (.1em);     
      
      \draw[thick] (0,0) -- (1,0.5);
      \draw[dashed] (1,0.5) -- (2,0.5);
      \draw[black,fill=black] (1,0.5) circle (.1em);
      \node[above] at (1,0.5) {$\eta_{+i,\frac{1}{2}}$}; 
      \draw[black,fill=black] (2,0.5) circle (.1em);
      \node[right] at (2,0.5) {$+i$}; 
      
      \draw[thick] (0,0) -- (1,-0.5);
      \draw[dashed] (1,-0.5) -- (2,-0.5);
      \draw[black,fill=black] (1,-0.5) circle (.1em);
      \node[below] at (1,-0.5) {$\eta_{-i,\frac{1}{2}}$}; 
      \draw[black,fill=black] (2,-0.5) circle (.1em);
      \node[right] at (2,-0.5) {$-i$}; 
      
      \node[below] at (1,-2.5) {$X_{\mathbb{Q}_3[i]}$};
      \draw [<->] (2.1,-0.4) to [out=45,in=315] (2.1,0.4);
      \end{scope}
    \end{tikzpicture}
  \end{center}
\caption{Forme modérément ramifiée mais non triviale de couronne}
\label{fig:Forme non triviale de couronne}
\end{figure}

\section*{Remerciements} 
Je tiens à faire part de toute ma gratitude à mon directeur Antoine Ducros, pour son soutien indéfectible et les innombrables discussions que nous avons eu au cours de l'élaboration de cet article.
Je suis aussi très reconnaissant envers Lorenzo Fantini et Daniele Turchetti pour m'avoir toujours tenu au courant de leurs résultats sur les couronnes \cite{FantiniTurchetti17} et nos échanges subséquents.

\printbibliography

\end{document}